\documentclass[11pt]{article}
\usepackage{amsmath,amsfonts,amssymb,amsthm}
\usepackage{mathtools,mathabx}
\usepackage{comment}
\usepackage{colortbl}
\RequirePackage[colorlinks,citecolor=blue,urlcolor=blue,backref=page]{hyperref}

\usepackage{colortbl}

\newtheorem{proposition}{Proposition}[section]
\newtheorem{theorem}[proposition]{Theorem}
\newtheorem{corollary}[proposition]{Corollary}
\newtheorem{lemma}[proposition]{Lemma}
\newtheorem{remark}[proposition]{Remark}

\newtheorem{example}[proposition]{Example}

\newcommand{\nc}{\newcommand}
\nc{\I}{{\mathbf 1}}
\nc{\bN}{{\mathbf N}}
\nc{\bG}{{\mathbf G}}
\nc{\bM}{{\mathbf M}}
\nc{\cB}{{\mathcal B}}
\nc{\cM}{{\mathcal M}}
\nc{\cG}{{\mathcal G}}
\nc{\R}{{\mathbb R}}
\nc{\N}{{\mathbb N}}
\nc{\Z}{{\mathbb Z}}
\nc{\BX}{{\mathbb X}}
\nc{\BY}{{\mathbb Y}}
\nc{\BM}{{\mathbb M}}
\nc{\cX}{{\mathcal X}}
\nc{\cY}{{\mathcal Y}}
\nc{\cN}{{\mathcal N}}
\nc{\cF}{{\mathcal F}}
\nc{\rX}{{\mathfrak X}}

\nc{\BH}{\mathbb{H}}

\DeclareMathOperator*{\esssup}{\mathrm{ess\, sup}}

\nc{\BP}{\mathbb{P}}
\nc{\BE}{\mathbb{E}}
\nc{\BQ}{\mathbb{Q}}

\numberwithin{equation}{section}

\begin{document} 

\renewcommand{\thefootnote}{\fnsymbol{footnote}}
\author{{\sc Mikhail Chebunin\footnotemark[1]} \hspace{0.1mm} and {\sc G\"unter Last\footnotemark[2]}} 
\footnotetext[1]{mikhail.chebunin@uni-ulm.de, 
Institute of Stochastics, Ulm University, 89069 Ulm, Germany. 
}
\footnotetext[2]{guenter.last@kit.edu,  
Institute of Stochastics, Karlsruhe Institute of Technology, 
76131 Karlsruhe, Germany. 
}

\title{On strong sharp phase transition in the random connection model} 
\date{\today}
\maketitle

\begin{abstract} 
\noindent 
We consider a random connection model (RCM) $\xi$ driven by a Poisson process $\eta$. We derive exponential moment bounds for an arbitrary cluster, provided that the intensity $t$ of $\eta$ is below a certain critical intensity $t_T$. The associated subcritical regime is characterized by a finite mean cluster size, uniformly in space. Under an exponential decay assumption on the connection function, we also show that the cluster diameters are exponentially small as well.
In the important stationary marked case and under
a uniform moment bound on the connection function,
we show that $t_T$ coincides with $t_c$, the largest $t$ for which $\xi$ does not percolate.
In this case, we also derive some percolation mean field bounds.  These findings generalize some of the main results in \cite{CaicedoDickson24}.
Even in the classical unmarked case, our results are more general than what has been previously 
known. Our proofs are partially based on some stochastic monotonicity properties, which might be 
of interest in their own right. 
\end{abstract}


\vspace{0.1cm}

\section{Introduction}\label{intro}

Let $(\BX,d)$ be a complete separable metric space, 
denote its Borel-$\sigma$-field by $\cX$, and let $\lambda$ be
a locally finite and diffuse measure on $\BX$.
Let $t\in\R_+:=[0,\infty)$ be an intensity parameter
and let $\eta$ be a Poisson process on $\BX$ with intensity
measure $t\lambda$, defined over a probability space 
$(\Omega,\cF,\BP)$. We often write $\BP_t$ instead
of $\BP$ and $\BE_t$ for the associated expectation operator. 

Let $\varphi\colon\BX^2\to[0,1]$ be a measurable
and symmetric function satisfying
\begin{align}\label{evarphiintlocal}
D_\varphi(x):=
\int \varphi(x,y)\,\lambda(dy)<\infty, \quad \lambda\text{-a.e.} \ x\in\BX.
\end{align}
We refer to $\varphi$ as {\em connection function}.
The {\em random connection model} (RCM)
is the random graph $\xi$ whose vertices are the points of $\eta$
and where a pair of distinct points $x,y\in\eta$ 
forms an edge with probability $\varphi(x,y)$, independently for
different pairs.  In an Euclidean setting, the RCM was introduced in
\cite{Pen91} (see \cite{MeesterRoy} for a textbook treatment), 
while the general Poisson
version was studied in \cite{LastNestSchul21}.  The RCM is a
fundamental and versatile example of a spatial random graph.  
Of particular interest is the {\em stationary  marked} case,
where $\varphi$ is translation invariant in the spatial coordinate. Special cases are the  Boolean model (see \cite{LastPenrose18, SW08}) with
general compact grains and the so-called weighted RCM; see 
\cite{BetschLast21,CaicedoDickson24, DicksonHeydenreich22, GLM21, GHMM22, HHLM23,
LastZiesche17, Pen16}.

Following common terminology of percolation theory, we refer to a
component of $\xi$ as {\em cluster}. The RCM $\xi$ {\em percolates},
if it has an infinite cluster, that is a component with infinitely
many vertices. Generalizing many earlier results, it was shown in
\cite{ChLast25} that the stationary marked RCM can have at most one
infinite cluster, provided a natural irreducibility assumption
holds. 
Take $v\in\BX$ and add independent connections
between $v$ and the points from $\eta$. 
Let $C^v$ denote the component of $v$ in this augmentation of $\xi$.
The  {\em critical intensity} $t_c$ is defined by
\begin{align}\label{e:tc}
t_c:=\sup\{t\ge 0: \BP_t(|C^v|<\infty)=0\; \text{for $\lambda$-a.e.\ $v$}\},
\end{align}
where $|C^v|$ stands for the number of vertices of $C^v$.
A second critical intensity is defined by
\begin{align}\label{e:tT}
t_T:=\sup\big\{t\ge 0: \esssup_{v\in\BX}\BE_t |C^v|<\infty\big\},
\end{align}
where the essential supremum refers to $\lambda$.  It is clear that $t_T\le t_c$. By a {\em sharp phase transition} it is usually understood that for $t<t_c$ the clusters are not only finite
but have a finite first moment in some suitable sense.
If $t<t_T$, we establish here uniform exponential moment properties for the
size and, under an additional necessary assumption on the connection
function, also for the diameter of a cluster.
Therefore, we refer to the identity $t_c=t_T$ as {\em strong sharp phase transition}.
If $\xi$ is a stationary marked RCM,  then we prove a strong sharp phase transition
under an integrability assumption on $\varphi$, which is only slightly stronger than the one required for $t_T>0$.
Under this assumption we also show that
$t_c$ is the smallest intensity where the mean size of a typical cluster is infinite. This identity might
serve as a definition of a sharp phase transition for the stationary marked RCM.
We also provide lower bounds for two critical
exponents. Our method of proving the strong sharpness of the phase
transition will be based on transferring some of the beautiful ideas
from the seminal paper \cite{AizBar87} in a way similar to
\cite{CaicedoDickson24}.

In the following we shall present two of our main results in greater detail along with a discussion of the 
relevant literature.  We will show in Lemma \ref{l:not_triv}  that
$t_T>0$ if and only if
\begin{align}\label{e:1009}
D^*_\varphi:=\esssup\limits_{v\in\BX}D_\varphi(v)<\infty
\end{align}
and that $t_T\ge (D^*_\varphi)^{-1}$.
The following result is the content of our Theorem \ref{tmoments_2}.

\begin{theorem}\label{t:intro1}  
For $t<t_T$ there exists $\delta_1\equiv\delta_1(t)$ such that $\esssup\limits_{v\in\BX}\BE_t e^{\delta_1 |C^v|}<\infty$.
\end{theorem}

We would like to emphasize that Theorem \ref{t:intro1} applies as soon as $t_T>0$, without making any further
assumptions on the state space or the connection function.
An important special case of a RCM is the stationary marked RCM. 
In that case $\eta$ is a Poisson process on $\BX:=\R^d\times\BM$ with intensity
measure $t\lambda_d\otimes\BQ$, where $\BM$ is a complete separable metric space,
$\lambda_d$ denotes Lebesgue measure on $\R^d$ and  $\BQ$ is a probability measure on $\BM$.
The connection function $\varphi$ is assumed to be translation invariant in the sense of
\eqref{e:4.1}. 
In the {\em unmarked case} $|\BM|=1$ it was proved in \cite{Meester95}
that $t_T= t_c$, provided that the (in a sense) minimal integrability
assumption $\int\varphi(0,x)\,dx<\infty$ is satisfied 
and assuming that $\varphi(0,x)$ is a non-increasing function of
the  Euclidean norm of $x$.  It was shown in \cite{KuePenrose24} 
that the cluster size has an exponential tail, provided the connection function is supported by the unit ball. 
Afterwards, this strong sharp phase transition
was derived in \cite{Higgs25} for an isotropic (and monotone) 
connection function with possibly unbounded support. Even in this very special case
of an (unmarked) stationary RCM,
Theorem \ref{t:intro1} extends all these results without making any assumption
other than the integrability of $\varphi$.
In the marked case, the situation is considerably more
complicated. To state a second main result of our paper, 
we define $d_\varphi(p,q):=\int\varphi((0,p),(x,q))\,dx$ for $p,q\in\BM$.
Our Theorem \ref{t:main} (and Remark \ref{r:suff})
shows the following. 
To avoid trivialities, we assume that $\int d_\varphi(p,q)\,\BQ^2(d(p,q))>0$.

\begin{theorem}\label{t:intro2} Suppose that $\xi$ is a stationary marked random connection model
and assume that 
\begin{align}\label{e:sharp}
\esssup_{p\in\BM}\int d_\varphi(p,q)^2\,\BQ(dq)<\infty,
\end{align} 
where the essential supremum refers to $\BQ$.
Then $t_T=t_c\in(0,\infty)$ and $\int \BE_{t_c} |C^{(0,p)}|\,\BQ(dp)=\infty$.
\end{theorem}

Note that \eqref{e:sharp} is (slightly) stronger than \eqref{e:1009},
implicitly assumed in Theorem \ref{t:intro1}.  Under the uniform
second moment assumption \eqref{e:sharp}, Theorems \ref{t:intro1} and
\ref{t:intro2} imply a strong sharp phase transition at $t_c=t_T$.  
At first glance it may be surprising that \eqref{e:sharp} yields uniform exponential moment
bounds for the cluster sizes in the subcritical regime $t<t_c$. 
This phenomenon can be understood by observing that assumption \eqref{e:1009} guarantees
that the degree distribution of the vertices has finite exponential moments of every order, uniformly in space. Note that, in particular, the classical
stationary RCM undergoes a strong sharp phase transition, without any assumption other than
integrability of the connection function. The same holds in the marked case, provided $\BQ$ is supported by a finite set. For the Boolean model with
(deterministically) bounded grains (covered by Theorems \ref{t:intro1}
and \ref{t:intro2}) it was proved in \cite{Ziesche18} that the size of
a typical cluster is exponentially small for $t<t_c$. This is a
continuum analogue of classical results in \cite{Men86} and
\cite{AizBar87}; see also \cite{DuminilTassion2016} for a new and
elegant proof (that inspired \cite{Ziesche18} and
\cite{KuePenrose24}).  Using a continuum version of the
OSSS-inequality the result from \cite{Ziesche18} was extended in
\cite{LastPeccYogesh23} so as to cover the so-called $k$-percolation.
Theorem \ref{t:intro2} generalizes \cite[Corollary
1.10]{CaicedoDickson24} by proving a sharp phase transition under a
weaker integrability assumption on the connectivity function and
without assuming any kind of irreducibility.  It also generalizes a strong sharp phase transition result from \cite{Pabst25}, which
assumes the connection function to have a bounded support in space
uniformly in the marks.  Moreover, our results extend \cite[Theorem
1.8]{CaicedoKolesnikov25} (uploaded after the first arXiv-version of
this paper) for a so-called min-reach RCM, by relaxing the exponential
moment assumption on their reach function to a finite moment of order
$2d$; see Example \ref{ex:min-reach}. Both
\cite{Pabst25} and \cite{CaicedoKolesnikov25} are using the discrete OSSS-inequality.

Condition \eqref{e:sharp} fails for a spherical Boolean model
with a radius distribution of unbounded support; see Example \ref{ex:Gilbert}.
To get from \eqref{e:sharp} to weaker conditions for 
a sharp phase transition at $t_c$, remains a challenging
problem. In the case of the spherical Boolean model,  
an important step was made in \cite{Duminiletal2020}, 
treating an unbounded radius distribution with a finite $d$-th moment. 
In this case $t_c>0$ while $t_T=0$ (see also Example \ref{ex:Gilbert}). 
Nevertheless, it is proved there in the subcritical regime
that under the polynomial (resp.\ exponential) decay behavior of the tail of the radius distribution,
certain connection probabilities show a similar behavior.
Using an approach from \cite{Gou08,GouTheret19} it was shown in \cite{DembinTassion22} that even for Pareto distributed radii (with minimal moment assumption)
the model undergoes such a {\em subcritical sharpness}.
Related results for the {\em soft Boolean model} with Pareto distributed weights (a special case of Example \ref{ex:weighted}) can be found  in \cite{JaLueOrt24}.

The paper is organized as follows. In Section \ref{s:definitions} we give the formal definition of the RCM $\xi$ and introduce the notation used throughout the paper.
Section \ref{s:basics} presents the RCM version of the multivariate Mecke equation,
while Section \ref{s:perc} introduces some basic notation and concepts of percolation theory. Section \ref{secMarkov}
presents a spatial Markov property (taken from \cite{ChLast25}) and some results on stochastic ordering, 
which might be of some independent interest.
Section \ref{secsub} provides a criterion for subcriticality and for
exponential moments of cluster sizes, while Section \ref{secDiameter}
deals with cluster diameters.
In Section \ref{s:stationaryRCM} we discuss the stationary marked RCM,
with Theorem \ref{t:main} summarizing our main results.
Proposition \ref{p:susc_mean_field_bound} presents a mean field lower bound
for mean cluster sizes in the subcritical case, while Theorem \ref{t:per_mean_field_bound}
provides such a bound for percolation probabilities in the supercritical case.

\section{Formal definition of the RCM}\label{s:definitions}

To define a RCM we follow \cite{ChLast25}.  Let
$\bN$ ($\bN_{<\infty}$) denote the space of all simple locally finite
(finite) counting measures on $\BX$, equipped with the standard
$\sigma$-field, see e.g.\ \cite{LastPenrose18}.  A measure $\nu\in\bN$
is identified with its support $\{x\in\BX:\nu(\{x\})=1\}$ and
describes the set of vertices of a (deterministic) graph.  If
$\nu(\{x\})=1$ we write $x\in\nu$. 
Any $\nu\in\bN$ can be written as a finite or
infinite sum $\nu=\delta_{x_1}+\delta_{x_2}+\cdots$ of Dirac measures, 
where the $x_i$ are pairwise distinct and do not accumulate in bounded sets.  The
space of (undirected) graphs with vertices from $\BX$ (and no loops)
is described by the set $\bG$ of all counting measures $\mu$ on
$\BX\times\bN$ with the following properties. First, we assume that the
measure $V(\mu):=\mu(\cdot\times\bN)$ is locally finite and simple,
that is, an element of $\bN$.  We write $x\in\mu$ if  $x\in V(\mu)$ (that is
$\mu(\{x\}\times\bN)=1$). In this case, there is a unique $\psi_x\in\bN$ such
that $(x,\psi_x)\in\mu$. We assume that $x\notin \psi_x$. Finally, if
$x\in V(\mu)$ and $y\in\psi_x$ then we assume that $(y,\psi_y)\in \mu$
and $x\in\psi_y$. Also, $\bG$ is equipped with the standard
$\sigma$-field. There is an edge between $x,y\in V(\mu)$ if $y\in\psi_x$ (and hence $x\in\psi_y$). If $\psi_x=0$, then $x$ is {\em isolated}.

We write $|\mu|:=\mu(\BX\times\bN)$ for the cardinality of $\mu\in\bG$ and similarly for $\nu\in\bN$. Hence $|\mu|=|V(\mu)|$. 
For $x,y\in V(\mu)$ we write $x\sim y$ (in $\mu$) if there is an edge between $x$ and $y$
and  $x \leftrightarrow y$ (in $\mu$) if there is a path in $\mu$ leading from $x$ to $y$.
For $A\subset\BX$ we write $x\sim A$ (in $\mu$) if there exists $y\in A\cap V(\mu)$
such that $x\sim y$.
Let $\mu,\mu'\in \bG$. Then $\mu$ is a {\em subgraph} of $\mu'$ if
$V(\mu)\le V(\mu')$ (as measures) and for each $(x,\psi)\in \mu$ and
$(x,\psi')\in\mu'$ we have $\psi\le \psi'$. Note that this is not the same
as $\mu\le\mu'$.

Let $\chi$ be a simple point process on $\BX$, which is a random element of
$\bN$. The reader should think of
a Poisson process possibly augmented by additional (deterministic) points.
By \cite[Proposition 6.2]{LastPenrose18} there exist random elements $X_1,X_2,\ldots $ of
$\BX$ such that $\chi=\sum^{|\chi|}_{n=1}\delta_{X_n}$,
where $X_m\ne X_n$ whenever $m\ne n$ and $m,n\le  |\chi|$.
Let $(Z_{m,n})_{m,n\in\N}$ be a double sequence of random elements 
uniformly distributed on $[0,1]$ such that $Z_{m,n}=Z_{n,m}$ for all
$m,n\in\N$ and such that $Z_{m,n}$, $m<n$, are independent.
Then the RCM (based on $\chi$) is the point process
\begin{align*}
\xi:=\sum^{|\chi|}_{m=1}\delta_{(X_m,\Psi_m)},
\end{align*}
where 
\begin{align*}
\Psi_m:=\sum^{|\chi|}_{n=1}\I\{n\ne m,Z_{m,n}\leq\varphi(X_m,X_n)\}\delta_{X_n}.
\end{align*} 
While the definition of $\xi$ depends on the ordering of the points of $\chi$, its distribution does not.

We now introduce some notation used throughout the paper.
For $\mu,\mu'\in\bG$, we often interpret $\mu+\mu'$ as the
measure in $\bG$ with the same support as $\mu+\mu'$.
A similar convention applies to $\nu,\nu'\in\bN$.
Let $\mu\in\bG$. For $B\in\cX$ we write $\mu(B):=\mu(B\times \bN)$.
More generally, given a measurable function $f\colon \BX\to \R$ we write 
$\int f(x)\,\mu(dx):=\int f(x)\,\mu(dx\times \bN)$. 
In the same spirit, we write $g(\mu):=g(V(\mu))$, whenever
$g$ is a mapping on $\bN$.
These (slightly abusing) conventions lighten the notation and should not cause any confusion. 
For $B\in\cX$ 
we denote by $\mu[B]\in\bG$ the {\em restriction} of $\mu$ to $B$, that is  the graph
with vertex set $V(\mu)\cap B$ which keeps only those edges
from $\mu$ with both end points from $B$. In the same way, we use the notation $\mu[\nu]$ for $\nu\in\bN$. 
Similarly, for a measure $\nu$ on $\BX$ (for instance for $\nu\in\bN$) 
we denote by $\nu_B:=\nu(B\cap\cdot)$  the restriction of $\nu$ to a set $B\in\cX$.
Assume now that $\mu$ is a subgraph of $\mu'$. For $n\in\N_0$ let
$C^\mu_n(\mu')\in\bG$ denote the restriction of $\mu'$ to  those $v\in V(\mu')$ with
$d_{\mu'}(v,\mu)=n$, where $d_{\mu'}$ denotes the distance within the graph $\mu'$.  
Note that $C^\mu_0(\mu')$ is just the graph $\mu$. Slightly abusing our notation,
we write $C^\mu_0(\mu')=\mu$ and $V^\mu_n(\mu')=V(C^\mu_n(\mu'))$. For $v\notin V(\mu')$ we set $C^v(\mu'):=0$,
interpreted as an empty graph (a graph without vertices).
The cluster $C^\mu(\mu')$  of $\mu$ in $\mu'$ is the graph $\mu'$ restricted to
\begin{align*}
V^\mu(\mu')=\sum^\infty_{n=0}V^\mu_n(\mu'),
\end{align*}
while $C^\mu_{\le{n}}(\mu')$, $n\in\N_0$, is the graph $\mu'$ restricted to 
$V^\mu_0(\mu')+\cdots+V^\mu_n(\mu')$. For later purposes, it will be convenient to define $C^\mu_{\le{-1}}(\mu')=C^\mu_{-1}(\mu'):=0$ as the zero measure. Throughout we write
$V^{V(\mu)}_n(\mu'):=V^{\mu}_n(\mu')$ and
$V^{\mu}_n(\mu',\cdot):=V^{\mu}_n(\mu')(\cdot)$, $n\in\N_0$, and
similarly for $V^{\mu}_{\le n}$ and $V^{\mu}$. We also often refer to $C^v(\mu')$ as the {\em cluster} of $v$ in $\mu'$ for $v\in V(\mu')$.

Given an RCM based on a Poisson process $\eta$ on 
$\BX$ with diffuse intensity measure $\lambda$,
we use the following notation. For $v\in\BX$ and $n\in\N_0$ we set
\begin{align*}
    C^v:=C^v(\xi^v), \quad V^v:=V^v(\xi^v), \quad C^v_n:=C^v_n(\xi^v), \quad V^v_n:=V^v_n(\xi^v), 
\quad C^v_{\le n}:=C^v_{\le n}(\xi^v), \quad V^v_{\le n}:=V^v_{\le n}(\xi^v),
\end{align*}
or $C^v_\lambda$, $V^v_\lambda$, $C^v_{n,\lambda}$, $V^v_{n,\lambda}$, $C^v_{\le n,\lambda}$ and $V^v_{\le n,\lambda}$, if we need to emphasize the dependence on $\lambda$. 
Moreover, we write $V^{v!}_{\le n}:=V^v_{\le n}-\delta_v$ and similarly for $V^v$.

\section{The Mecke equation}\label{s:basics}

Let $\xi$ be a RCM based on a Poisson process $\eta$ on 
$\BX$ with diffuse intensity measure $\lambda$.
Our first crucial tool is a version of the Mecke equation (see
\cite[Chapter 4]{LastPenrose18}) for  $\xi$.  Given $n\in\N$ and 
$\boldsymbol{x}_n:=(x_1,\ldots,x_n)\in\BX^{n}$ we denote
$\delta_{\boldsymbol{x}_n}:=\delta_{x_1}+\cdots+\delta_{x_n}$ and
$\eta^{\boldsymbol{x}_n}:=\eta+\delta_{\boldsymbol{x}_n}$ (removing possible multiplicities)
and let $\xi^{\boldsymbol{x}_n}$ denote a RCM based on $\eta^{\boldsymbol{x}_n}$.
It is useful to construct $\xi^{\boldsymbol{x}_n}$ in a specific way as follows.
We connect $x_1$ with the points in $\eta$ using independent connection decisions which
are independent of $\xi$.
We then proceed inductively, finally connecting $x_n$ to 
$\eta+\delta_{\boldsymbol{x}_{n-1}}$.
For a measurable function
$f\colon\BX^{n} \times \bG\to [0,\infty]$ the Mecke equation for $\xi$ states that 
\begin{align}\label{e:Mecke}
\BE \int f(\boldsymbol{x}_n,\xi)\,\eta^{(n)}(d \boldsymbol{x}_n)
= \BE \int f(\boldsymbol{x}_n,\xi^{\boldsymbol{x}_n})\,\lambda^{n}(d \boldsymbol{x}_n),
\end{align}
where integration with respect to  the {\em factorial measure} $\eta^{(n)}$
of $\eta$ means summation over all $n$-tuples of pairwise distinct points
from $\eta$. 

For given $v\in\BX$ and $\boldsymbol{x}_n\in\BX^n$
we denote $(v,\boldsymbol{x}_{n}):=(v,x_1,\ldots,x_n)\in\BX^{n+1}$. 
We sometimes use \eqref{e:Mecke} in the form
\begin{align}\label{e:Meckev}
\BE \int f(\boldsymbol{x}_n,\xi^v)\,\eta^{(n)}(d \boldsymbol{x}_n)
= \BE \int f(\boldsymbol{x}_n,\xi^{v,\boldsymbol{x}_{n}})\,\lambda^{n}(d \boldsymbol{x}_n).
\end{align}
The proofs of \eqref{e:Mecke} and \eqref{e:Meckev} can be found in \cite{ChLast25}.

\section{Percolation and critical intensities}\label{s:perc}
\subsection{Notation and terminology in the general case}\label{sub:generalnotation}

Let $t\ge 0$ be an intensity parameter
and let $\xi$ be a RCM based on a Poisson process $\eta$ on $\BX$ with an intensity measure $t\lambda$, where $\lambda$ is a locally finite and diffuse measure on $\BX$. 
The RCM $\xi$ {\em percolates} if it has an infinite cluster, a component with infinitely many vertices.
We also say that the RCM (or $t$) is {\em subcritical} if all clusters have only a finite number of points, that is,
\begin{align*}
\BP_t\big(|V^v|<\infty\big)=1,\quad \lambda\text{-a.e.\ $v\in\BX$}.
\end{align*}
In accordance with \eqref{e:tc}
we define the {\em critical intensity} $t_c$ as the supremum of all $t\in\R_+$ such that above holds. 
A standard coupling argument shows that $\xi$ is subcritical for all $t<t_c$.

Let $v\in\BX$, $t\ge 0$ and $n\in\N_0$. Mean generations and cluster sizes are denoted by
\begin{align*}
\begin{split}
    c^v_n(t):=\BE_t |V^v_n|, \quad
    c^v_{\le n}(t):=\BE_t |V^v_{\le n}|, \quad c^v(t):=\BE_t |V^v|.
\end{split}
\end{align*}
It is clear that $c^v_1(t)=t D_\varphi(v)$ and
\begin{align*}
    c^v(t)=\sum_{n=0}^\infty c_n^v(t)=\lim_{n\to\infty} c^v_{\le n}(t).
\end{align*}
For a measurable function $L\colon \BX \to\R\cup\{ \pm \infty\}$
we define the $\infty$-norm by $\| L\|_{\infty}:=\esssup_{v\in\BX} |L(v)|$,
where the essential supremum refers to $\lambda$. 
We abbreviate
\begin{align}\label{*notation}
D^*_{\varphi}:= \|D_\varphi\|_{\infty}, \quad c^*_n(t):=\| c_n(t)\|_{\infty},
\quad c^*_{\le n}(t):=\| c_{\le n}(t)\|_{\infty}, \quad c^*(t):=\|c(t)\|_{\infty}.
\end{align}
The {\em second critical intensity} $t_T$ is defined as the
supremum of all $t\in\R_+$ such that $c^*(t)<\infty$. It is clear that
$t_T\le t_c$. 

\subsection{The stationary marked RCM}\label{sub:notationstatRCM}

In this subsection, we consider the important special case of a stationary RCM; see e.g.\ \cite{CaicedoDickson24, DicksonHeydenreich22,ChLast25}. 
Let $\BM$ be a complete separable metric space (the mark space)
equipped with a probability measure $\BQ$ (the mark distribution).
Set $\BX:=\R^d\times\BM$ be equipped with the
product of the Borel $\sigma$-field $\cB(\R^d)$ on $\R^d$ and the
Borel $\sigma$-field on $\BM$. We assume that
$\lambda=\lambda_d\otimes\BQ$, where $\lambda_d$ denotes the Lebesgue measure on $\R^d$.  If $(x,p)\in\BX$ then we call $x$ location of $(x,p)$ and $p$ the mark of $x$. Instead of $\bN$ we consider the
(smaller set) $\bN(\BX)$ of all counting measures $\chi$ on $\BX$ such
that $\chi(\cdot\times\BM)$ is locally finite (w.r.t.\ the Euclidean
metric) and simple.  
We can and will assume that the Poisson process $\eta$ is a random element of $\bN(\BX)$. 

The symmetric connection function
$\varphi\colon (\R^d\times\BM)^2\to [0,1]$ is assumed to satisfy
\begin{align}\label{e:4.1}
\varphi((x,p),(y,q))= \varphi((0,p),(y-x,q)).
\end{align}
This allows us to write $\varphi(x,p,q):=\varphi((0,p),(x,q))$, 
where $0$ denotes the origin in $\R^d$. 
The RCM $\xi$ is {\em stationary} in the sense
that $T_x\xi\overset{d}{=}\xi$, $x\in\R^d$, where
for $\mu\in\bG$, the measure $T_x\mu$ is (shift of $\mu$ by $x$) defined by 
$T_x\mu:=\int \I\{(y-x,q,\nu)\in\cdot\}\,\mu(d(y,q,\nu))$.
It is also ergodic; see \cite{ChLast25}.

\begin{remark}\label{r:432}\rm
The argument in \cite{ChLast25} 
can be extended to yield that $T_x\xi^{(x,p)}\overset{d}{=}\xi^{(0,p)}$ for $\lambda_d\otimes\BQ$-a.e.\ 
$(x,p)\in\R^d\times\BM$.
Hence, if $f\colon\bG\to\R$ is measurable and shift invariant, then
$f(\xi^{(x,p)})\overset{d}{=}f(\xi^{(0,p)})$ for
$\lambda_d\otimes\BQ$-a.e.\  $(x,p)\in\R^d\times\BM$. Therefore, the definitions \eqref{e:tc}
and \eqref{e:tT} of $t_c$ and $t_T$ can be simplified. For instance, we have
\begin{align}\label{e:tcmarked}
t_c=\sup\{t\ge 0: \BP_t(|C^{(0,p)}|<\infty)=0\; \text{for $\BQ$-a.e.\ $p$}\}.
\end{align}
\end{remark}

The function $D_\varphi$ (defined by \eqref{evarphiintlocal}) takes the form
\begin{align*}
D_\varphi((x,p))=\iint \varphi(y,p,q)\,dy\,\BQ(dq),\quad (x,p)\in\R^d\times\BM,
\end{align*}
while $D^*_\varphi$ is given by
\begin{align*}
\esssup_{p\in\BM}\iint \varphi(y,p,q)\,dy\,\BQ(dq),
\end{align*}
where the essential supremum now refers to $\BQ$.
Similar comments apply to other characteristics introduced in Subsection \ref{sub:generalnotation}. 
For instance, we have
$c^*(t)=\esssup_{p\in\BM}c^{(0,p)}(t)$. We often write
\begin{align}
\bar{c}_n(t):=\int c^{(0,p)}_n(t)\,\BQ(dp),\quad \bar{c}_{\le n}(t):=\int c^{(0,p)}_n(t)\,\BQ(dp),\quad
\bar{c}(t):=\int c^{(0,p)}(t)\,\BQ(dp).
\end{align}
It is convenient to introduce a random element $Q_0$ of $\BM$ which is independent of $\xi$.
and has distribution $\BQ$. Then we denote by $C^{Q_0}$ the cluster of $(0,Q_0)$
in the RCM $\xi^{Q_0}$ arising from $\xi$ by adding independent connections between $(0,Q_0)$ and
the points from $\eta$. This
is the cluster of a {\em typical vertex} and we let $V^{Q_0}$ denote its vertex set.
Then we can write $\bar{c}(t)=\BE_t |V^{Q_0}|$ and similar for other quantities.

Define
\begin{align*}
\theta^{p}(t):=\BP_t\big(\big|C^{(0,p)}\big|=\infty\big), \quad t\ge 0, \ p\in\BM,
\end{align*}
as the probability that the cluster of a vertex $(0,p)\in\BX$ has
infinite size.  In the following, we use the $L^r(\BQ)$-norms to determine the size of functions. For each
$r\in\left[1,+\infty\right]$, we define the critical intensities
\begin{align}
    t_c^{(r)} &:= \inf\left\{t\geq 0: \|\theta(t)\|_r >0\right\},\label{e:percr}\\
    t_T^{(r)} &:= \inf\left\{t\geq 0: \|c(t)\|_r =\infty\right\}.\label{e:meancr}
\end{align}
It is clear that $t_c^{(r)}$ is $r$-independent, since it only matters if
$\theta^{p}(t)=0$ for $\BQ$-a.e.\ $p\in\BM$ or not. Therefore, all
the critical intensities $t_c^{(r)}$ coincide with the first critical
intensity $t_c$, which we defined for a general RCM. Note that
$t_T^{(r)} \leq t_c$ for all $r$. Moreover, $t^{(\infty)}_T$ coincides
with the second critical intensity $t_T$, which we have also defined for
a general RCM. From Jensen's inequality, it is clear that $t_T^{(r)}$ is non-increasing in
$r$. Therefore,
\begin{equation}
    0\leq t_T=t_T^{(\infty)} \leq t_T^{(r_2)} \leq t_T^{(r_1)} \leq t_T^{(1)} \leq t_c,\quad 1\leq r_1\leq r_2 \leq \infty.
\end{equation}

Let
\begin{align*}
\bar\theta(t):= \| \theta(t)\|_1=\int \theta^p(t)\,\BQ(dp), \quad t\ge 0,
\end{align*}
denote the percolation probability, that is, the probability that the cluster of a typical vertex is
infinite. Let $C_\infty$ be the set of all $\mu\in\bG$ such that
$\mu$ has an infinite cluster. 
It is easy to see that $\bar\theta(t)>0$ iff $\BP_t(\xi\in C_\infty)=1$.
In fact, if $t<t_c$ then $\BP_t(\xi\in C_\infty)=0$ and if $t>t_c$ then $\BP_t(\xi\in C_\infty)=1$.
Under a natural irreducibility assumption, $\xi$ can have at most one infinite cluster; 
see \cite{ChLast25}.

\begin{remark}\label{ex:classical}\rm In the unmarked case ($|\BM|=1$) the
connection function $\varphi$ is just a function on $\R^d$. 
Under the minimal assumption $0<\int\varphi(x)\,dx<\infty$
it was shown in \cite{Pen91} that $t_c\in(0,\infty)$. 
\end{remark}

The marked RCM is a very rich and flexible model of a spatial random graph.
We refer to \cite{CaicedoDickson24,ChLast25,GHMM22} for many examples.
For further reference, we provide just two of them here.

\begin{example}\label{ex:Gilbert}\rm Assume that $\BM=\R_+$ and
$\varphi(x,p,q)=\I\{\|x\|\le p+q\}$, where $\|x\|$ denotes the Euclidean norm of $x\in\R^d$.
The RCM $\xi$ is known as the spherical Boolean model or as {\em Gilbert graph} with
radius distribution $\BQ$; see e.g.\ \cite[Chapter 16]{LastPenrose18} for more detail.
We have that
\begin{align}
D_\varphi((0,r))=\kappa_d\int (r+s)^d\,\BQ(ds),\quad r\ge 0, 
\end{align}
where $\kappa_d$ stands for the volume of the unit ball. Therefore
our basic assumption \eqref{evarphiintlocal} is
equivalent with $\int r^d\,\BQ(dr)<\infty$.
This is the minimal assumption for having a reasonable model.
Under the additional assumption $\BQ\{0\}<1$ it was proved in \cite{Gou08,Hall85} that $t_c\in(0,\infty)$.
On the other hand, if $\BQ$ has unbounded support, then $D^*_\varphi=\infty$ and
$t_T=0$; see Lemma \ref{l:not_triv}.
\end{example}

\begin{example}\label{ex:weighted}\rm Assume that $\BM=(0,1)$ 
equipped with Lebesgue measure $\BQ$. Assume that
\begin{align*}
\varphi((x,p),(y,q))=\rho(g(p,q)\|x-y\|^d),
\end{align*}
for a profile function $\rho\colon [0,\infty)\to [0,1]$ and a kernel function $g\colon(0,1)\times(0,1)\to [0,\infty)$.
We assume that $m_\rho:=\int \rho(\|x\|^d)\,dx$ is positive and finite. This model was studied in \cite{GHMM22}
under the name {\em weight-dependent random connection model}, for decreasing profile function and increasing kernel function.
Then
\begin{align*}
\iint \varphi(x,p,q)\,dx\,dq=m_\rho\int g(p,q)^{-1}\,dq,
\end{align*}
and \eqref{evarphiintlocal} holds if $g(p,\cdot)^{-1}$ is integrable for each $p\in(0,1)$. 
This is the case in all examples studied in \cite{GHMM22}, where it is also asserted that $t_c<\infty$.
Sufficient conditions for $t_c\in(0,\infty)$ can also be found in \cite{CaicedoDickson24,DepW15}. 
\end{example}

\section{A spatial Markov property and stochastic ordering}\label{secMarkov}

We consider a general RCM $\xi$ based on a Poisson process $\eta$ on $\BX$ with diffuse intensity measure $\lambda$.
Given $\mu\in\bN$, we denote the RCM based on $\eta+\mu$ by $\xi^\mu$.
We first recall the {\em spatial Markov property}, as formulated in \cite{ChLast25}.

Let $\nu$ be a locally finite and diffuse
measure on $\BX$. We often write $\Pi_\nu$ for the distribution of a Poisson process with intensity measure $\nu$. 
We set $\bar\varphi:=1-\varphi$ and define for $x\in\BX$, $\nu\in\bN$
\begin{align}\label{e:barvarphi}
\bar\varphi(\nu,x):=\prod_{y\in\nu}\bar\varphi(x,y), \quad \varphi(\nu,x):=1-\bar\varphi(\nu,x), \
\quad \varphi_\lambda(\nu):=\int \varphi(\nu,x)\,\lambda(dx).
\end{align}
We recall our general convention $\varphi(\mu,x):=\varphi(V(\mu),x)$ 
and $\varphi_\lambda(\mu):=\varphi_\lambda(V(\mu))$ for $\mu\in\bG$.
Next we define two kernels from $\bN$ to $\BX$ and from $\bN\times\bN$ to $\BX$ (using the same notation $K_\nu$ for simplicity), by 
\begin{align}
K_\nu(\mu,dx):=\bar\varphi(\mu,x)\nu(dx), \quad K_\nu(\mu,\mu',dx):=\bar\varphi(\mu,x)\varphi(\mu',x)\nu(dx).
\end{align} 
Proposition \ref{p2.1} will provide an interpretation of this kernel. 
Denoting by $0$ the zero measure, we note that
\begin{align}
K_\nu(0,dx)=\nu(dx), \quad K_\nu(0,\mu',dx)=\varphi(\mu',x)\nu(dx), \quad K_\nu(\mu,0,dx)=0.
\end{align}
We write $K_\nu(\mu):=K_\nu(\mu,\cdot)$ and $K_\nu(\mu,\mu'):=K_\nu(\mu,\mu',\cdot)$.
Note that $K_\lambda(0,\mu,\BX)=\varphi_\lambda(\mu)$; see \eqref{e:barvarphi}.

The following spatial Markov property of the random graph $\xi^v$ 
was proved in \cite{ChLast25}.

\begin{proposition}\label{p2.1} The sequence $(V^v_{\le{n-1}},V^v_n)_{n\in\N_0}$ is a Markov
process with transition kernel
\begin{align*}
(\mu,\mu')\mapsto \int\I\{(\mu+\mu',\psi)\in\cdot\}\,\Pi_{K_\lambda(\mu, \mu')}(d\psi).
\end{align*}
\end{proposition}

In the following, we often abbreviate $K:=K_\lambda$.
Given $n\in\N$ we define a probability kernel $H_n$ from $\bN\times\bN$ to $\bN$ by 
\begin{align}\label{e:kernelHn}
H_n(\mu,\mu',\cdot):=\idotsint\I\{\psi_n\in\cdot\}\, \Pi_{K(\mu+\mu'+\psi_1+\cdots+\psi_{n-2},\psi_{n-1})}(d\psi_n)\cdots
\Pi_{K(\mu+\mu',\psi_1)}(d\psi_2)\Pi_{K(\mu,\mu')}(d\psi_1).
\end{align}
The Proposition \ref{p2.1} implies that
\begin{align}\label{distVn}
\BP(V^v_n\in\cdot)=H_n(0,\delta_v,\cdot),\quad v\in\BX.
\end{align}

\begin{corollary}\label{c3.23} For $\lambda$-a.e. $v\in\BX$ and $n\in\N_0$ under condition \eqref{evarphiintlocal} we have $\BP(|V^v_n|<\infty)=1$.
\end{corollary}

The following useful property of the kernel $K_\lambda$ can be easily proved by induction.

\begin{lemma}\label{l3.1} Let  $n\in\N$ and $\mu_0,\ldots,\mu_n\in\bN$.
Then
\begin{align*}
K_\lambda(0,\mu_0)+K_\lambda(\mu_0,\mu_1)+\cdots+K_\lambda(\mu_0+\cdots+\mu_{n-1},\mu_{n})
=K_\lambda(0,\mu_0+\cdots+\mu_{n}).
\end{align*}
\end{lemma}

Using a standard coupling argument, it is easy to establish the following facts.

\begin{proposition}\label{p2.2}
    Let $v\in\BX$ and assume that $\lambda_1\le \lambda_2$. Then 
    \begin{align}\label{Corder}
        V^v_{\lambda_1}\le_{st} V^v_{\lambda_2},
    \end{align}
    and for $n\in\N$
    \begin{align}\label{Clenorder}
        V^v_{\le n,\lambda_1}\le_{st} V^v_{\le n, \lambda_2}.
    \end{align}
\end{proposition}
\begin{proof}
    Construct $\xi^v_{\lambda_2}$ and then $\xi^v_{\lambda_1}$ by independent thinning. 
\end{proof}

The Proposition \ref{p2.2} implies monotonicity in the measure $\lambda$ for
$V^v_{\le n}$ and $V^v$. We can state similar property for the first and second generations, but things are getting tricky for higher
generations. This is an open question.

\begin{proposition}\label{p2.2+}
Let $v\in\BX$, $m\in\{1,2\}$ and assume that $\lambda_1\le \lambda_2$. Then
    \begin{align}\label{Cmorder}
        V^v_{m,\lambda_1}\le_{st}  V^v_{m,\lambda_2}.
    \end{align}
\end{proposition}
\begin{proof}
Construct $\xi^v_{\lambda_2}$ and then $\xi^v_{\lambda_1}$ by
independent thinning. Then
$ V^v_{1,\lambda_1}\le V^v_{1,\lambda_2}$ a.s.. 
The Proposition \ref{p2.1} shows that, given
$V^v_{1,\lambda_2}$, the conditional distribution of
$V^v_{2,\lambda_2}$ is that of a Poisson process with an intensity measure $K_{\lambda_2}(\delta_v, V^v_{1,\lambda_2})$. The statement follows from the facts that a Poisson process stochastically increases in
the intensity measure and that $K_\lambda$ is increasing in $\lambda$ and the second coordinate, i.e.
  $K_{\lambda_2}(\delta_v, V^v_{1,\lambda_2})\ge
  K_{\lambda_1}(\delta_v, V^v_{1,\lambda_1})$.
\end{proof}

\begin{lemma}\label{l:st_ineq_V_n_mu} Suppose that
$\mu_1,\mu_2 \in\bN_{<\infty}$ have disjoint support and set
$\mu:=\mu_1+\mu_2$. Assume that $\xi_{\mu_1}$ and $\xi_{\mu_2}$ are independent RCMs with distributions $\xi^{\mu_1}$ and $\xi^{\mu_2}$, respectively.
Then
\begin{align}\label{ineq:st_V_n_mu0}
V^\mu_n(\xi^\mu)\le_{st} V_n^{\mu_1}(\xi_{\mu_1})+ V_n^{\mu_2}(\xi_{\mu_2}),\quad n\in\N_0,\\
\label{ineq:st_V_n_mu1}
V^\mu_{\le n}(\xi^\mu)\le_{st} V_{\le n}^{\mu_1}(\xi_{\mu_1})+ V_{\le n}^{\mu_2}(\xi_{\mu_2}),\quad n\in\N_0.
\end{align}
\end{lemma}
\begin{proof}
Define $V_{1,1}(\xi^\mu):=\{ x\in\eta : x \sim \mu_1 \}$ as the set
of points from $\eta$ which are connected to $\mu_1$ in $\xi^\mu$, and
$V_{1,2}(\xi^\mu):=\{ x\in\eta : x \sim \mu_2, \ x\nsim \mu_1  \}$. 
Then $V_{1,1}(\xi^\mu)$ and $V_{1,2}(\xi^\mu)$ are independent Poisson processes, and
$V_1^\mu(\xi^\mu)= V_{1,1}(\xi^\mu)+V_{1,2}(\xi^\mu)$. 
For $k\in\N$ we define
$V_{k,1}(\xi^\mu)$ as the set of points from $V^\mu_k(\xi^\mu)$ which are
connected to $V_{k-1,1}(\xi^\mu)$ and $V_{k,2}(\xi^\mu)$ 
as the set of points from $V^\mu_k(\xi^\mu)$ which are connected to
$V_{k-1,2}(\xi^\mu)$ but not to $V_{k-1,1}(\xi^\mu)$. 
Then   $V_k^\mu(\xi^\mu)= V_{k,1}(\xi^\mu)+V_{k,2}(\xi^\mu)$. Moreover,
$V_{k,1}(\xi^\mu)$ and $V_{k,2}(\xi^\mu)$ are conditionally independent Poisson
processes given $C^\mu_{\le k-1}(\xi^\mu)$. 
Let $f\colon \bN \to \R_+$  be measurable and $n\in\N$. Then
\begin{align*}
&\BE f(V^\mu_n(\xi^\mu)) = 
\BE\, \BE \big[ f\left(V^\mu_{n}(\xi^\mu)\right) \mid C^\mu_{\le n-1}(\xi^\mu)\big] \\
 &=\BE \iint  f\left(\psi^1_n+ \psi^2_n\right) \,  
\Pi_{K(V^\mu_{\le n-2}(\xi^\mu)+V_{n-1,1}(\xi^\mu),V_{n-1,2}(\xi^\mu))} (d \psi^2_n)
\,\Pi_{K(V^\mu_{\le n-2}(\xi^\mu),V_{n-1,1}(\xi^\mu))} (d \psi^1_n).
\end{align*}
Recursively, we get
\begin{align}\notag\label{eq:mean_V_n_mu} 
&\BE f(V^\mu_n(\xi^\mu))\\ \notag 
&=\idotsint  f\big(\psi^1_n+\psi^2_n\big) \, \Pi_{K(\mu+\psi^1_1+\psi^2_1+\cdots+\psi^1_{n-2}+\psi^2_{n-2}  +\psi^1_{n-1},\psi^2_{n-1})} (d \psi^2_n)
\, \Pi_{K(\mu+\psi^1_1+\psi^2_1+\cdots+\psi^1_{n-2}+\psi^2_{n-2},\psi^1_{n-1})} (d \psi^1_n) \\ 
        &\qquad \qquad\cdots \Pi_{K(\mu+\psi^1_1,\psi^2_1)}(d \psi^2_2) \,  \Pi_{K(\mu,\psi^1_1)} (d \psi^2_1)
        \, \Pi_{K(\mu_1,\mu_2)} (d \psi^2_1) \, \Pi_{K(0,\mu_1)} (d \psi^1_1). 
    \end{align}
On the other hand, we obtain from \eqref{distVn} and the independence of $\xi_{\mu_1}$ and $\xi_{\mu_2}$ that
\begin{align}\label{eq:mean_V_n_mu12}
\BE f\left( V^{\mu_1}_n(\xi_{\mu_1})+V^{\mu_2}_n(\xi_{\mu_2})\right)
=\iint f\big(\psi^1_n+\psi^2_n\big) \,H_n(0,\mu_1,d\psi^1_n)\,H_n(0,\mu_2,d\psi^2_n).
\end{align}
Assume now that $f$ increases. We compare \eqref{eq:mean_V_n_mu} and
\eqref{eq:mean_V_n_mu12} taking into account two facts. 
First, Poisson processes increase stochastically in the intensity measure. Second,
$K$ decreases in the first argument and increases in
the second. This implies
\begin{align*}
\BE f(V^\mu_n(\xi^\mu))\le  \BE f\left( V^{\mu_1}_n(\xi_{\mu_1})+V^{\mu_2}_n(\xi_{\mu_2})\right), 
\end{align*}
that is \eqref{ineq:st_V_n_mu0}. The proof of \eqref{ineq:st_V_n_mu1} is the same, up to the
fact that the argument of the function $f$ has to be suitably modified.
\end{proof}

\begin{corollary}\label{c:st_ineq_V_n_mu}
Let $\mu\in\bN_{<\infty}$ and $n\in\N_0$. Then
\begin{align*}
V^\mu_n(\xi^\mu)&\le_{st} \int V_n^x(\xi_x) \, \mu(d x),\\
V^\mu_{\le n}(\xi^\mu)&\le_{st} \int V_{\le n}^x(\xi_x) \, \mu(d x).
\end{align*}
where $\xi_x$, $x\in\mu$, are independent and $\BP(\xi_x\in\cdot)=\BP(\xi^x\in\cdot)$.
\end{corollary}
\begin{proof} Write $\mu=\delta_{x_1}+\cdots+\delta_{x_m}$ and apply
Lemma \ref{l:st_ineq_V_n_mu} inductively with 
$\mu_1=\delta_{x_1}+\cdots+\delta_{x_i}$ and
$\mu_2=\delta_{x_{i+1}}$ for $i=1,\ldots,m-1$.
\end{proof}

In the following, we consider two families $\{\xi_x^{(n)}\}_{x\in V^v_n}$ 
and $\{\xi^{[n]}_x\}_{x\in V^v_n}$ of random graphs with the following properties: 
\begin{itemize}
\item[1.] Given $C_{\le n}^v$, $\xi_x^{(n)}$ and $\xi^{[n]}_x$ are RCM's driven
  by $\eta_n+\delta_x$ and $\eta+\delta_x$, where $\eta_n$ and $\eta$
  are Poisson processes with intensity measures
  $K_\lambda(V^v_{\le n-1})$ and $\lambda$ respectively.
    \item[2.]  The members of the families are conditionally independent, given $C_{\le n}^v$. 
\end{itemize}

\begin{proposition}\label{p2.3} Let $v\in\BX$ and $n, m\in\N_0$. Then 
\begin{align*}
    V_{n+m}^v \le_{st}  \int V^x_m(\xi_x^{(n)})\,V^v_n(dx).
\end{align*}
\end{proposition}
\begin{proof} Let $f\colon \bN \to \R_+$ be measurable and increasing.
By Proposition \ref{p2.1} the conditional distribution of $V_{n+m}^v$ 
given $C_{\le n}^v$ is that of $V_m^{V^v_n}(\xi^{(V^v_{\le n-1}),V^v_n})$, where $\xi^{(V^v_{\le n-1}),V^v_n}$ 
is a RCM based on $\eta_n+V_n^v$ and $\eta_n$ 
is a Poisson process with intensity measure $K_\lambda(V^v_{\le n-1})$.  Hence, 
by Corollary \ref{c:st_ineq_V_n_mu}
\begin{align*}
\BE f(V^v_{n+m}) = \BE\, \BE\big[f(V^v_{n+m}) \mid C^v_{\le n}\big] 
= \BE\, \BE \Big[f(V_m^{V^v_n}(\xi^{(V^v_{\le n-1}),V^v_n})) \mid C^v_{\le n}\Big] 
\le  \BE f\left(  \int V^x_m(\xi_x^{(n)})\,V^v_n(dx) \right).
\end{align*}
The assertion follows.
\end{proof}

For $v\in\BX$ and $n\in\N_0$ we define the {\em intensity measures} $\Lambda^v_n$ 
of $V^v_n$ by $\Lambda^v_n(B):=\BE V^v_n(B)$ for $B\in\cX$.  

\begin{proposition}\label{p2.3+} Let $v\in\BX$ and $n\in\N_0$, $m\in\{1,2\}$. Then 
\begin{align*}
V_{n+m}^v \le_{st}  \int V^x_m(\xi^{[n]}_x)\,V^v_n(dx).
\end{align*}
Moreover, we get
\begin{align*}
\Lambda^v_{n+m}\le\int \Lambda^x_m\,\Lambda^v_n(dx).
\end{align*}
\end{proposition}
\begin{proof}
Let $f\colon\bN \to \R_+$ be measurable and increasing. By Propositions \ref{p2.3} and \ref{p2.2+},
\begin{align*}
        \BE f(V^v_{n+m}) = \BE \BE (f(V^v_{n+m}) \mid C^v_{\le n}) \le 
        \BE f\left(  \int V^x_m(\xi_x^{(n)})\,V^v_n(dx) \right)
        \le \BE f\left(  \int V^x_m(\xi_x^{[n]})\,V^v_n(dx) \right).
\end{align*}
Since $\Lambda^v_0=\delta_v$ for the second statement, we can assume that $n\in\N$. For each $B\in\cX$ we have 
\begin{align*}
\Lambda^v_{n+m}(B)=\BE \BE(V^v_{n+m}(B) \mid C^v_{\le n}) 
\le \BE \int \Lambda^x_m(B)\,V^v_n(dx)
= \int \Lambda^x_m(B)\,\Lambda^v_n(dx),
\end{align*}
where we have used Campbell's formula.
\end{proof}

\begin{proposition}\label{p2.4}
    Let $v\in\BX$ and $n,m\in\N$. Then 
\begin{align*}
V^v_{\le n+m}\le_{st} V^v_{\le n-1}+\int V^x_{\le m}(\xi^{(n)}_x)\,V^v_n(dx)
\le_{st} V^v_{\le n-1}+\int V^x_{\le m}(\xi^{[n]}_x)\,V^v_n(dx).
\end{align*}
\end{proposition}
\begin{proof} Let $f\colon \bN \to \R_+$ be measurable and increasing.
By Proposition \ref{p2.1} the conditional distribution of $V_{\le n+m}^v$ given $C_{\le n}^v$ is 
that of $V_{\le n-1}^v + V_{\le m}^{V^v_n}(\xi^{(V^v_{\le n-1}),V^v_n})$, where $\xi^{(V^v_{\le n-1}),V^v_n}$ is a RCM 
based on $\eta_n+V_n^v$ and $\eta_n$ is a Poisson process with 
intensity measure $K_\lambda(V^v_{\le n-1})$.  Hence, we obtain from
Corollary \ref{c:st_ineq_V_n_mu} and Proposition \ref{p2.2}
\begin{align*}
\BE f(V^v_{\le n+m}) = \BE\, \BE\big[f(V^v_{\le n+m}) \mid C^v_{\le n}\big]
& \le  \BE f\left( V^v_{\le n-1} + \int V^x_{\le m}(\xi_x^{(n)})\,V^v_n(dx) \right) \\
&\le \BE f\left(V^v_{\le n-1} +  \int V^x_{\le m}(\xi^{[n]}_x)\,V^v_n(dx) \right).
\end{align*} 
\end{proof}

\begin{corollary}\label{c:clus_dom}
Let $v\in\BX$ and $n\in\N$. Then 
\begin{align*}
V^v\le_{st} V^v_{\le n-1}+\int V^x(\xi^{[n]}_x)\,V^v_n(dx).
\end{align*}
\end{corollary}
\begin{proof}
The claim follows from Proposition \ref{p2.4}, since with probability one $V^v_{\le n+m}\uparrow V^v$ and 
$V^x_{\le m}(\xi^{[n]}_x)\uparrow V^x(\xi^{[n]}_x)$ as $m\to\infty$.
\end{proof}

Given $v\in\BX$ and a measurable function $h\colon \BX\to\N$, 
we define two spatial (Galton-Watson) branching processes  $(W^{v,h}_k)_{k\ge 0}$ and $(\widetilde W^{v,h}_k)_{k\ge 0}$ 
along with a sequence of families of random graphs $\{\xi^k_x:x\in W^{v,h}_k\}$, $k\in\N_0$,
recursively as follows. We set 
\begin{align*}
    W^{v,h}_0=\widetilde W^{v,h}_0=\delta_v, \quad  W^{v,h}_1= V^v_{h(v)}, \quad \widetilde W^{v,h}_1= V^{v!}_{\le h(v)},
\end{align*}
and $\{\xi^0_x:x\in W^{v,h}_0\}:=\{\xi^v\}$.
Given $k\ge 1$, $Z_k:=\big(\big(W^{v,h}_n,\widetilde W^{v,h}_n\big)\big)_{n\le k}$ and $\big(\{\xi^n_x:x\in W^{v,h}_n\}\big)_{n\le k-1}$
we let $\{\xi^k_x:x\in W^{v,h}_k\}$ be a family of random graphs which are conditionally independent given $Z_k$ and $\BP$-a.s.
\begin{align*}
    \BP(\xi^k_x\in \cdot \mid Z_k)=\BP (\xi^x\in \cdot),  \quad x\in W^{v,h}_{k}.
\end{align*}  
Then we define
\begin{align*}
W^{v,h}_{k+1}:= \int V^x_{h(x)} (\xi^k_x) \, W^{v,h}_{k}(d x), \quad 
\widetilde W^{v,h}_{k+1}:= \int V^{x!}_{\le h(x)} (\xi^k_x) \, W^{v,h}_{k}(dx.) 
\end{align*}
Note that
\begin{align*}
\widetilde W^{v,h}_{k+1}= \int V^{x!}_{\le h(x)-1} (\xi^k_x) \, W^{v,h}_{k}(d x)+W^{v,h}_{k+1}.
\end{align*}

Hence, for every point $x\in W^{v,h}_k$ from the $k$-th
generation we run an independent RCM driven by
$\eta +\delta_x$, where $\eta$ is a Poisson process with intensity
measure $\lambda$ and place all points of its $h(x)$-th generation in
our spatial branching process $k+1$-th generation, i.e. $W^{v,h}_{k+1}$.  
The process $\big(\widetilde W^{v,h}_k\big)_{k\ge 0}$ has a similar interpretation.
We define
\begin{align*}
 W^{v,h}_{\le k} = \sum_{i=0}^k W^{v,h}_i, \quad  W^{v,h}=\sum_{i=0}^\infty W^{v,h}_i, \quad   
\widetilde W^{v,h}_{\le k} = \sum_{i=0}^k \widetilde W^{v,h}_i, \quad   \widetilde W^{v,h} = \sum_{i=0}^\infty \widetilde W^{v,h}_i.
\end{align*}
The special case $h\equiv n$ for some fixed $n\in\N$ is of particular importance. In that
case we use the upper index $n$ instead of $h$.

\begin{remark}\rm \label{r:loc_bound}
For each fixed $k\in\N$, the point process $\widetilde W^{v,h}_{\le k}$ is locally
finite. If the extinction probability of $W^{v,h}$ is equal to one, then this is also
true for $\widetilde W^{v,h}$. Otherwise, one cannot be sure of the locally
bounded property. Corollary \ref{c:subcrit}  
provides a simple sufficient condition 
(namely $D^*_\varphi<\infty$ and $c^*_n(t)< 1$ for some $n\in\N$) when $W^{v,n}$ is finite with probability
one. In the stationary case, it suffices to assume that $c^*_n(t)\le 1$,
since the total size of $W^{v,n}$ coincides with the total size of a
Galton–Watson process with offspring distribution $|V^0_n|$.
\end{remark}

We have the following useful property.

\begin{proposition}\label{p2.5}
Let $v\in\BX$ and $k,n\in\N$. Then 
\begin{align}\label{e876}
W^{v,n}_{k}&\le_{st} W^{v,1}_{kn},\\
\label{e877}
W^{v,n}_{\le k}&\le_{st} \sum_{i=1}^k W^{v,1}_{ i n}.
    \end{align}
In particular, if $k=1$, then
\begin{align}\label{e876_1}
    V^{v}_{n}&\le_{st} W^{v,1}_{n}.
\end{align}
\end{proposition}
\begin{proof}
Define a kernel $\tilde{K}$ from $\bN$ to $\bN$ by
\begin{align}\label{tildeK}
\tilde K(\mu,\cdot):=\int K_\lambda(0,\delta_x,\cdot)\,\mu(dx).
\end{align}
By Bernoulli's inequality, we have
\begin{align}\label{Kmon}
K_\lambda(\mu,\mu',\cdot)\le \tilde K(\mu',\cdot),\quad \mu,\mu'\in\bN.
\end{align}
Taking an increasing and measurable $f\colon\bN\to\R$, we therefore obtain from \eqref{distVn}
and the monotonicity properties of a Poisson process that
\begin{align}
\BE f(V^v_n)\le
\idotsint f(\psi_n)\, \Pi_{\tilde K(\psi_{n-1})}(d\psi_n)\cdots
\Pi_{\tilde K(\psi_1)}(d\psi_2)\Pi_{\tilde K(\delta_v)}(d\psi_1)=\BE f(W^{v,1}_n).
\end{align}
This is \eqref{e876} for $k=1$, i.e. \eqref{e876_1}.

By the definition of $W^{v,n}_{2}$ and \eqref{distVn} we have
\begin{align}\label{2944}
\BE f(W^{v,n}_{2})=\BE \int f(\psi)H'_n(V^v_n,d\psi), 
\end{align}
where for given $m\in\N$ and $\delta_{\boldsymbol{x}_m}\in\bN_{<\infty}$
\begin{align*}
H'_n(\mu,\cdot):=\idotsint \I\{\psi^1+\cdots+\psi^m\in\cdot\}\,H_n(0,\delta_{x_1},d\psi^1)\cdots \,H_n(0,\delta_{x_m},d\psi^m). 
\end{align*}
Using \eqref{Kmon} and the monotonicity properties of a Poisson process
in the definition \eqref{e:kernelHn} of kernels $H_n$ once again, we see that 
\begin{align*}
\int f(\psi)H'_n(\mu,d\psi)
&\le \idotsint f(\psi^1_n+\cdots+\psi^m_n)\,\Pi_{\tilde K(\psi^1_{n-1})}(d\psi^1_n)\cdots \Pi_{\tilde K(\psi^m_{n-1})}(d\psi^m_n)\times\cdots\\
&\qquad\qquad \times \Pi_{\tilde K(\delta_{x_1})}(d\psi^1_1)\cdots \Pi_{\tilde K(\delta_{x_m})}(d\psi^m_1).
\end{align*}
By the definition \eqref{tildeK} of the kernel $\tilde K$ and the fundamental properties of a Poisson process, the above $m$ innermost integrals equal
\begin{align*}
\int f(\psi_n)\,\Pi_{\tilde K(\psi^1_{n-1}+\cdots+\psi^m_{n-1})}(d\psi_n).
\end{align*}
Proceeding inductively, we obtain that
\begin{align*}
\int f(\psi)H'_n(\mu,d\psi)\le \idotsint f(\psi_n)\, \Pi_{\tilde K(\psi_{n-1})}(d\psi_n)\cdots \Pi_{\tilde K(\mu)}(d\psi_1).
\end{align*}
The above right-hand side is an increasing function of $\mu$. Hence, we can apply
\eqref{e876} for $k=1$ to obtain from \eqref{2944}
\begin{align*}
\BE f(W^{v,n}_{2})\le \idotsint f(\psi_n)\, \Pi_{\tilde K(\psi_{n-1})}(d\psi_n)\cdots \Pi_{\tilde K(\chi_n)}(d\psi_1)
\Pi_{\tilde K(\chi_{n-1})}(d\chi_n)\cdots \Pi_{\tilde K(\delta_v)}(d\chi_1)=\BE f(W^{v,1}_{2n}).
\end{align*}
Hence, we have \eqref{e876} for $k=2$. The case of a general $k$ can be treated analogously.
The second assertion \eqref {e877} can be proved in the same way, modifying
the arguments of the function $f$ in an appropriate way.
\end{proof}

In the following result, and also later, we consider a Galton-Watson process $(W_k)_{k\ge 0}$  
with Poisson offspring distribution with parameter 
$D^*_\varphi=\|D_\varphi\|_\infty$ starting with $W_0=1$. Then $W_k$ is the number of points in 
the $k$-th generation.  We set $W_{\le k}:=\sum_{i=0}^k W_i$.

\begin{proposition}\label{p:GW_c_1} 
Let $v\in\BX$ and $n\in\N$. Then 
    \begin{align*}
        |W^{v,1}_n|\le_{st} W_n, \quad |W^{v,1}_{\le n}|\le_{st} W_{\le n}.
    \end{align*}
\end{proposition}
\begin{proof}
Note that the total size of the offspring distribution $|V^x_1(\xi_x)|$ is a Poisson random variable with parameter 
$D_\varphi(x)$. Recall the definition \eqref{tildeK}.  
Taking an increasing and measurable $f\colon\N\to\R$, we 
obtain from the monotonicity properties of a Poisson process that
\begin{align*}
\BE f(|W^{v,1}_n|)=&
\idotsint f(|\psi_n|)\, \Pi_{\tilde K(\psi_{n-1})}(d\psi_n)\cdots
\Pi_{\tilde K(\psi_1)}(d\psi_2)\Pi_{\tilde K(\delta_v)}(d\psi_1)\\
\le& \idotsint f(|\psi_n|)\, \Pi_{D^*_\varphi\psi_{n-1}}(d\psi_n)\cdots
\Pi_{D^*_\varphi\psi_1}(d\psi_2)\Pi_{D^*_\varphi\delta_v}(d\psi_1)\le \BE f(W_n).
\end{align*}

The proof of the second inequality is the same, up to the fact that the argument of the function $f$ has to be suitably modified.
\end{proof}

The following results are crucial for later purposes.

\begin{proposition}\label{p2.6}
Let $v\in\BX$ and $n, k\in\N$. Then 
\begin{align*}
V^v_{\le kn}\le_{st} \widetilde W^{v,n}_{\le k}.
\end{align*}
\end{proposition}
\begin{proof}
By Proposition \ref{p2.4}, from the similar definitions of $\{\xi_x^{[n]} : x\in V^v_n\}$ and $\{\xi_x^{1} : x\in V^v_n\}$ we get 
\begin{align}\label{6076}
V^v_{\le kn}\le_{st} V^v_{\le n-1}+\int V^{x_1}_{\le (k-1)n}\big(\xi^1_{x_1}\big)\,V^v_n(d x_1).
\end{align}
Define a point process $V^{[2]}_n$ on $\BX \times\BX$ by 
\begin{align*}
V^{[2]}_n:=\iint \I\{(x_1,x_2)\in\cdot\}\,V^{x_1}_{ n}(\xi^1_{x_1}, d x_2)\,V^v_n(d x_1).
\end{align*}
Let $\big\{\xi^{[2]}_{x_1,x_2}:(x_1,x_2)\in V^{[2]}_n\big\}$ be a family of random graphs which are conditionally independent given $V^{[2]}_n$ and $\BP$-a.s.
\begin{align*}
    \BP(\xi^{[2]}_{x_1,x_2}\in \cdot \mid V^{[2]}_n)=\BP (\xi^{x_2}\in \cdot),  \quad (x_1,x_2)\in V^{[2]}_n.
\end{align*}  
This implies that for each $x_1\in V^v_n$ the random graphs
$\xi^{[2]}_{x_1,x_2}$, $x_2\in V^{x_1}_{ n}(\xi^1_{x_1})$, are conditionally
independent given $V^{x_1}_{n}(\xi^1_{x_1})$ and
\begin{align*}
    \BP(\xi^{[2]}_{x_1,x_2}\in \cdot \mid V^{x_1}_{n}(\xi^1_{x_1}) )=\BP (\xi^{x_2}\in \cdot),  \quad x_2\in V^{x_1}_{n}(\xi^1_{x_1}).
\end{align*}  
Therefore, we can apply Proposition \ref{p2.4} in \eqref{6076} to the conditional distribution
(w.r.t.\ $V^{[2]}_n$) to see that
\begin{align*}
V^v_{\le kn} \le_{st}& V^v_{\le n-1}+\int V^{x_1}_{\le n-1}(\xi^1_{x_1}) \, V^v_n(d x_1)
+\iint V^{x_2}_{\le (k-2)n}(\xi^{[2]}_{x_1,x_2})\,V^{x_1}_n(\xi^1_{x_1}, d x_2)\,V^v_n(d x_1),
\end{align*}
It is easy to see that the distribution
of the above right-hand side does not change when replacing $\xi^{[2]}_{x_1,x_2}$ by $\xi^2_{x_2}$. It follows that
\begin{align*}
V^v_{\le kn} \le_{st}& V^v_{\le n-1}+\int V^{x_1}_{\le n-1}(\xi^1_{x_1}) \, V^v_n(d x_1)
+\iint V^{x_2}_{\le (k-2)n}(\xi^2_{x_2})\,V^{x_1}_n(\xi^1_{x_1}, d x_2)\,V^v_n(d x_1).
\end{align*}
Proceeding inductively, we obtain
\begin{align}\label{078}\notag
V^v_{\le kn} \le_{st} V^v_{\le n-1}&+\int V^{x_1}_{\le n-1}(\xi^1_{x_1}) \, V^{v}_n(d x_1)+\cdots +
\idotsint V^{x_{k-2}}_{\le n-1}(\xi^{k-2}_{x_{k-2}}) \, V^{x_{k-3}}_{n}(\xi^{k-3}_{x_{k-3}},d x_{k-2})\cdots V^{v}_n(d x_1)\\
&+\idotsint V^{x_{k-1}}_{\le n}(\xi^{k-1}_{x_{k-1}}) \, V^{x_{k-2}}_{n}(\xi^{k-2}_{x_{k-2}},d x_{k-1})\cdots V^{v}_n(d x_1).
\end{align}
Here, the last term can be written as
\begin{align*}
\idotsint V^{x_{k-2}}_n(\xi^{k-2}_{x_{k-2}}) &\, V^{x_{k-3}}_{n}(\xi^{k-3}_{x_{k-3}},d x_{k-2})\cdots V^{v}_n(d x_1)\\
&\quad +\idotsint V^{x_{k-1}!}_{\le n}(\xi^{k-1}_{x_{k-1}}) \, V^{x_{k-2}}_{n}(\xi^{k-2}_{x_{k-2}},d x_{k-1})\cdots V^{v}_n(d x_1)\\
&=\idotsint V^{x_{k-2}}_n(\xi^{k-2}_{x_{k-2}}) \, V^{x_{k-3}}_{n}(\xi^{k-3}_{x_{k-3}},d x_{k-2})\cdots V^{v}_n(d x_1)
+\widetilde{W}^{v,n}_k.
\end{align*}
Inserting this into \eqref{078}, performing the preceding step several times, we end up with
\begin{align*}
V^v_{\le kn} \le_{st} V^v_{\le n}+\widetilde{W}^{v,n}_1+\cdots+\widetilde{W}^{v,n}_k= \widetilde W^{v,n}_{\le k}.
\end{align*} 
This concludes the proof.
\end{proof}

\begin{corollary}\label{c2.1}
Let $v\in\BX$ and $k\in\N$. Then 
\begin{align*}
V^v_{\le k}\le_{st} W^{v,1}_{\le k}.
\end{align*}
\end{corollary}
\begin{proof}
The statement follows from Proposition \ref{p2.6} for $n=1$ and the observation 
that $\widetilde W^{v,1}_{\le k}\equiv W^{v,1}_{\le k}$.
\end{proof}

\begin{proposition}\label{p2.9}
Let $v\in\BX$ and $h\colon \BX\to\N$ be a measurable function.
Then 
\begin{align*}
 V^v\le_{st} \widetilde W^{v,h}.
\end{align*}
\end{proposition}
\begin{proof}  
As in the proof of Proposition \ref{p2.6} it follows for each $n\in\N$ that
\begin{align*}
V^v_{\le n} \le_{st}& V^v_{\le h(v)-1}+\int V^{x_1}_{\le h(x_1)-1}(\xi^1_{x_1}) \, V^v_{h(v)}(d x_1)
+\iint V^{x_2}_{\le n-h(v)-h(x_1)}(\xi^{[2]}_{x_1,x_2})\,V^{x_1}_{h(x_1)}(\xi^1_{x_1}, d x_2)\,V^v_{h(v)}(d x_1).
\end{align*}
Note that the distribution
of the above right-hand side does not change upon replacing
$\xi^{[2]}_{x_1,x_2}$ by $\xi^2_{x_2}$. It follows that
\begin{align*}
V^v_{\le n} \le_{st}& V^v_{\le h(v)-1}+\int V^{x_1}_{\le h(x_1)-1}(\xi^1_{x_1}) \, V^v_{h(v)}(d x_1)
+\iint V^{x_2}_{\le n-h(v)-h(x_1)}(\xi^2_{x_2})\,V^{x_1}_{h(x_1)}(\xi^1_{x_1}, d x_2)\,V^v_{h(v)}(d x_1).
\end{align*}
Proceeding inductively, since $h(x)\ge 1$ for any $x\in\BX$, we obtain for any $n\in\N$
\begin{align*}
    V^v_{\le n} \le_{st} \widetilde W^{v,h}_{\le n} \le \widetilde W^{v,h}.
\end{align*}
This concludes the proof, since with probability one $V^v_{\le n}\uparrow V^v$ as $n\to\infty$.
\end{proof}

\begin{corollary}\label{p2.8} 
Let $v\in\BX$ and $n\in\N$.
Then 
\begin{align*}
V^v\le_{st} \widetilde W^{v,n}.
 \end{align*}
\end{corollary}

\section{A criterion for subcriticality and exponential moments}\label{secsub}

We establish the setting of Section \ref{secMarkov} 
with intensity measure $\lambda$ replaced by $t\lambda$ for some $t\ge 0$. 
We use the so-called generations method to build lower bounds on $t_T$, see e.g. \cite{Men85,Zuev85,Men86,Men87,Men88}.
As is common in percolation theory, we denote the underlying probability measure by $\BP_t$, 
to stress the dependence on the intensity parameter. 

\begin{lemma}\label{l:not_triv}
We have $t_T\ge (D^*_\varphi)^{-1}$. Moreover, we have $t_T>0$ if and only if $D^*_\varphi<\infty$,
\end{lemma}
\begin{proof}
If $ D^*_\varphi=0$, then each Poisson point is isolated and $t_T=\infty$. If
$D^*_\varphi=\infty$, then $c^*_1(t)=t D^*_\varphi=\infty$ for any $t>0$ and $t_T=0$. 
Let $t\ge 0$ and assume that $0<D^*_\varphi<\infty$. 
The Proposition \ref{p2.5} implies that $c^v_n(t)\le \BE_t |W^{v,1}_n|$. By Proposition \ref{p:GW_c_1} we have
$\BE_t |W^{v,1}_n|\le \BE_t W_n=(c^*_1(t))^n$ for any $n\in\N$. Therefore,
$t_T\ge 1/D^*_\varphi>0$.
\end{proof}

In the remainder of the section, we shall assume that $ D^*_\varphi <\infty$.

\begin{remark}\rm
    One can attempt to improve the lower bound of $t_T$ by obtaining an estimate of the form 
\begin{align*}
    \BE_t (|V_{n+l}^v|\mid C_{\le n}^v)\le \gamma(t) |V_n^v|
\end{align*} 
for all $n\ge n_0$ and some $l, n_0 \in\N$. Then as soon as
$\gamma(t_0)<1$ there is no percolation and $t_T>t_0$. It is true that
in continuous models, obtaining such estimates for $l\ge 2$ involves
appreciable technical difficulties.
\end{remark}

By Proposition \ref{p2.3+} for $n\in\N_0$ we have
\begin{align}\label{esub}
c^v_{n+2}(t)\le c^*_2(t) {c}^v_n(t).
\end{align}
This easily implies that if $c^*_2(t)<1$, then $t<t_T$. One of the main goals of this section is 
to show that if $D^*_\varphi<\infty$ and $c^*_n(t)<1$ for some $n\in\N$, then $t<t_T$.

\begin{lemma}\label{lmoments} Let $n, k\in\N$ and assume that $D^*_\varphi<\infty$. 
Then for any $t, \delta>0$ 
\begin{align*}
\esssup_{v\in\BX} \BE_t e^{\delta |W^{v,n}_{\le k}|}<\infty.
\end{align*}
\end{lemma}
\begin{proof}
By Proposition \ref{p2.5}  we have $W_{\le k}^{v,n}\le_{st} W^{v,1}_{\le kn}$. Then by Proposition \ref{p:GW_c_1} we obtain 
\begin{align*}
    \esssup_{v\in\BX} \BE_t e^{\delta |W^{v,n}_{\le k}|}\le \BE_t e^{\delta W_{\le kn}}.
\end{align*}
Since the offspring distribution of $W_{\le kn}$ is Poisson with parameter $c^*_1(t)<\infty$, it is well-known that $W_{\le kn}$ 
has exponential moments of each order; see e.g. \cite{Nakayama04}.
\end{proof}

\begin{corollary}\label{cmoments} Assume that $D^*_\varphi<\infty$ and let $n\in\N$. Then for any $t, \delta>0$ 
\begin{align*}
\esssup_{v\in\BX} \BE_t e^{\delta |V_{\le n}^v|}<\infty.
\end{align*}
\end{corollary}
\begin{proof}
By Proposition \ref{p2.6} we have $V_{\le n}^v\le_{st} W^{v,1}_{\le n}$. Then the result follows from Lemma \ref{lmoments}.
\end{proof}

\begin{theorem}\label{t:subcrit}
    Let $h\colon \BX\to\N$ be a measurable bounded function and $t\ge 0$. 
Assume that $D^*_\varphi<\infty$ and $c^*_h(t):=\esssup\limits_{v\in\BX} c^v_{h(v)}(t)<1$. Then 
    \begin{align*}
        \esssup_{v\in\BX} \BE_t |\widetilde W^{v,h}|\le \frac{c^*_{\le h^*}(t)}{1-c^*_h(t)}<\infty,
    \end{align*}
    where $h^* := \|h\|_\infty$.
\end{theorem}
\begin{proof}
By Fubini's theorem, we have
\begin{align*}
        \BE_t |W^{v,h}|& = 1 + c_{h(v)}^v(t)+\sum_{k\ge 1} \BE_t \int |V^x_{h(x)}(\xi_x^k)| \, W^{v,h}_k(dx)
        = 1 + c_{h(v)}^v(t)+\sum_{k\ge 1} \BE_t \int c^x_{h(x)}(t) \, W^{v,h}_k(dx) \\
        &\le 1+c_{h(v)}^v(t)+ c^*_h(t) (\BE_t |W^{v,h}|-1).
\end{align*}
Therefore, $\BE_t |W^{v,h}|\le (1+c_{h(v)}^v(t)- c^*_h(t))/(1- c^*_h(t))\le (1- c^*_h(t))^{-1}$ for any $v\in\BX$. 
Moreover, by Corollary \ref{cmoments} and Fubini's theorem, we obtain that for any $v\in\BX$
\begin{align*}
\BE_t |\widetilde W^{v,h}|&= c_{\le h(v)}^v(t)+\sum_{k\ge 1} \BE_t \int |V^{x!}_{\le h(x)}(\xi_x^k)| \, W^{v,h}_k(dx)
        = c_{\le h(v)}^v(t)+\sum_{k\ge 1} \BE_t \int (c_{\le h(x)}^x(t)-1) \, W^{v,h}_k(dx) \\
        &\le c_{\le h(v)}^v(t)  + (c^*_{\le h^*}(t)-1)(\BE_t |W^{v,h}|-1)
       < \frac{ c^*_{\le h^*}(t)}{1- c^*_h(t)}< \infty.
\end{align*}
This proves the result.
\end{proof}

\begin{corollary}\label{c:subcrit}
    Let $n\in\N$ and $t\ge 0$. If $D^*_\varphi<\infty$ and $c^*_n(t)<1$, then 
$c^*(t)\le c^*_{\le n}(t)/(1-c^*_n(t))<\infty$ and $t<t_T$.
\end{corollary}
\begin{proof}
The assertion follows from Corollary \ref{p2.8} and Theorem \ref{t:subcrit} with $h\equiv n$. 
\end{proof}

In the following, we generalize Corollary \ref{c:subcrit} for the case where it is known at which point of the underlying space the mean cluster size takes its maximal value. 

\begin{proposition}[A special criterion for subcriticality]\label{p:spetial_crit}
Let $v_0\in\BX$, $n\in\N$ and $t\ge 0$. If $D^*_\varphi<\infty$, $c^{v_0}(t)=c^*(t)$ and $c_n^{v_0}(t)<1$, 
then $c^*(t)\le c^{v_0}_{\le n-1}(t)/(1-c^{v_0}_n(t))<\infty$ and $t<t_T$.
\end{proposition}
\begin{proof}
    By Corollary \ref{c:clus_dom} and Fubini's theorem, we have
    \begin{align*}
        c^*(t)=c^{v_0}(t)\le c^{v_0}_{\le n-1}(t) + \BE_t \int c^x(t) \, V^{v_0}_n(d x)\le c^{v_0}_{\le n-1}(t)+c^{v_0}_{n}(t)c^{v_0}(t).
    \end{align*}
    Therefore, $c^{v_0}(t)\le c^{v_0}_{\le n-1}(t)/(1-c^{v_0}_n(t))<\infty$.
\end{proof}

\begin{theorem}\label{t:subcrit_ex}
    Under the conditions of Theorem \ref{t:subcrit} there exists $\delta=\delta(t, h^*)>0$ such that
    \begin{align*}
       \esssup_{v\in\BX} \BE_t e^{\delta |\widetilde W^{v,h}|}<\infty.
    \end{align*}
\end{theorem}
\begin{proof}
To describe the total population sizes of $W^{v,h}$ and $\widetilde W^{v,h}$, we define an exploration process. 
Let's consider the following spatial Markov process $(X_k)_{k\ge 0}=(Y_k,Z_k)_{k\ge 0}$ 
along with a sequence of families of random graphs $\{\xi^k_x:x\in Z_k\}$, $k\in\N_0$, recursively as follows. We set 
\begin{align*}
    X_0=(\delta_v,\delta_v), \quad X_1=(V^{v}_{\le h(v)},V^{v}_{ h(v)}),
\end{align*}
 and $\{\xi^0_x:x\in Z_0\}:=\{\xi^v\}$. Given $k\ge 1$, $(X_n)_{n\le k}$ and $\big(\{\xi^n_x:x\in Z_n\}\big)_{n\le k-1}$
we let $\{\xi^k_x:x\in Z_k\}$ be a family of random graphs which are conditionally independent given $Z_k$ and $\BP$-a.s.
\begin{align*}
    \BP(\xi^k_x\in \cdot \mid Z_k)=\BP (\xi^x\in \cdot),  \quad x\in Z_{k}.
\end{align*}  
Then we define
\begin{align*}
    Y_{k+1}=&  Y_k + \int \I\{d(v,x)=\min_{y\in Z_k} d(v,y)\} V^{x!}_{\le h(x)}(\xi_x^k) \, Z_k(d x), \\
    Z_{k+1}=& Z_k+\int \I\{d(v,x)=\min_{y\in Z_k} d(v,y)\} (V^x_{h(x)}(\xi_x^k)-\delta_x) \, Z_k(d x).
\end{align*}
Let $\tau:=\min(k\ge 1 : Z_k=\emptyset)$. Note that $\tau$ is a
stopping time w.r.t. $\{Z_k\}_{k\in\N}$. Moreover, $\tau=|W^{v,h}|$
and $Y_\tau=\widetilde W^{v,h}$ in distribution. Then we can use the
well-known test function criteria to prove the existence of an
exponential moment for $\tau$ (see \cite[Corollary 2,
p. 115]{Kalashnikov94}). Let $\delta, \varepsilon>0$ and
$\tau_k:=\min(\tau,k)$ for $k\ge 1$. Then for any $k\ge 1$ with
probability one
    \begin{align*}
        e^{\delta \tau_k}\le e^{\delta \tau_k+\varepsilon|Z_{\tau_k}|}=e^{\varepsilon|Z_0|}+\sum_{m=0}^{\tau_k-1}\left( e^{\delta(m+1)+\varepsilon|Z_{m+1}|}- e^{\delta m+\varepsilon|Z_{m}|}\right).
    \end{align*}
    Note that by Corollary \ref{cmoments} we have as $\varepsilon\to 0$
    \begin{align*}
        \BE_t & \left( e^{\delta(m+1)+\varepsilon|Z_{m+1}|}- e^{\delta m+\varepsilon|Z_{m}|}\mid Z_{m}\right)\\ 
        &= e^{\delta m+\varepsilon|Z_{m}|} \BE_t  \left( \exp\left(\delta+\varepsilon \int \I\{d(v,x)=\min_{y\in Z_m} d(v,y)\} (|V^x_{h(x)}(\xi_x^m)|-1) 
\, Z_m(d x)\right)- 1\mid Z_{m}\right)\\
        &\sim e^{\delta m+\varepsilon|Z_{m}|}  \left( e^\delta\left(1+ \varepsilon \int \I\{d(v,x)=\min_{y\in Z_m} d(v,y)\} (c^x_{h(x)}(t)-1) 
\, Z_m(d x)\right)- 1\right) \\
        &\le e^{\delta m+\varepsilon|Z_{m}|} \left( e^\delta\left(1+ \varepsilon (c^*_h(t)-1)\right)- 1\right).
    \end{align*}   
   Then there exist $\varepsilon, \delta>0$ such that for any $m\ge 0$
   \begin{align*}
       \BE_t \left( e^{\delta(m+1)+\varepsilon|Z_{m+1}|}- e^{\delta m+\varepsilon|Z_{m}|}\right)=\BE_t \,  \BE_t \left( e^{\delta(m+1)+\varepsilon|Z_{m+1}|}- e^{\delta m+\varepsilon|Z_{m}|}\mid Z_{m}\right)\le 0.
   \end{align*}
Therefore, $\BE_t e^{\delta \tau_k}\le \BE_t e^{\varepsilon|Z_0|}=e^\varepsilon$ for any $k\ge 1$. Letting $k\to\infty$ we obtain the light tail property for $\tau$ uniformly in $v\in\BX$. 

Note that by Propositions \ref{p2.6} and \ref{p:GW_c_1} we have for any $k\ge 1$
\begin{align*}
    \BE_t e^{\delta |Y_k|}&=\BE_t \left(e^{\delta |Y_{k-1}|} \BE_t
    \left(\exp\left(\delta \int \I\{d(v,x)=\min_{y\in Z_{k-1}} d(v,y)\} |V^{x!}_{\le h(x)}(\xi_x^{k-1})|\, Z_{k-1}(d x)\right)\mid X_{k-1}\right)\right) \\
    &=\BE_t \left(e^{\delta |Y_{k-1}|} \int \I\{d(v,x)=\min_{y\in Z_{k-1}} d(v,y)\} \BE_t( e^{\delta |V^{x!}_{\le h(x)}(\xi_x^{k-1})|}\mid Z_{k-1})\, Z_{k-1}(d x)\right) \\
    &\le \BE_t e^{\delta |Y_{k-1}|} \BE_t e^{ \delta  (W_{\le  h^*}-1)} \le \left(\BE_t e^{ \delta  (W_{\le  h^*}-1)}\right)^k.
\end{align*}

    By Theorem \ref{t:subcrit} $|\widetilde W^{v,h}|$ is a proper random variable. Then, by Fubini's theorem, Cauchy–Schwarz inequality, and the previous inequality, we also have 
    \begin{align*}
    \BE_t e^{ \delta |\widetilde W^{v,h}| }&= \BE_t \left(e^{ \delta |\widetilde W^{v,h}| }\sum_{k=1}^\infty \I\{\tau=k\}\right)
        =\sum_{k=1}^\infty \BE_t (e^{\delta |Y_k|} \I\{\tau=k\}) \\
        &\le \sum_{k=1}^\infty \sqrt{\BE_t e^{2\delta |Y_k|} \BP_t(\tau=k)}
        \le \sum_{k=1}^\infty \sqrt{\big(\BE_t e^{2\delta (W_{\le h^*}-1)}\big)^k \BP_t(\tau=k)}.
    \end{align*}
Since $W_{\le  h^*}$ has exponential moments of all orders (see e.g. \cite{Nakayama04}) 
and $\lim\limits_{\delta\to 0} \BE_t e^{2\delta (W_{\le  h^*}-1)}=1$ we get the required result from the uniform light tail property for $\tau$. 
\end{proof}

\begin{corollary}\label{c:subcrit_ex}
 Under the conditions of Corollary \ref{c:subcrit} there exists $\delta\equiv\delta(t,n)>0$ such that
    \begin{align}\label{e:lightailed}
       \esssup_{v\in\BX} \BE_t e^{\delta |V^{v}|}<\infty.
    \end{align}
\end{corollary}
\begin{proof}
    The claim follows from Corollary \ref{p2.8} and Theorem \ref{t:subcrit_ex} with $h\equiv n$. 
\end{proof}

\begin{theorem}\label{tmoments_2}
Suppose that $t<t_T$. Then there exists $\delta\equiv\delta(t)>0$ such that
\eqref{e:lightailed} holds.
\end{theorem}
\begin{proof} Since $t<t_T$ we have $c^*_1(t)=t D^*_\varphi\le c^*(t)<\infty$. 
Define $N:=\lceil 2 c^*(t) \rceil$ and 
\begin{align*}
A_i:=\{ v\in\BX : c^{v}_i(t)\le 1/2\}\setminus \bigg(\bigcup_{j=1}^{i-1} A_j\bigg),\quad i\in\N.
\end{align*} 
Then $\lambda(\bigcap_{i=1}^{N} A^c_i)=0$, since otherwise $c^*(t)> N/2\ge c^*(t)$. 
Let $n$ be the smallest number such that $\lambda(\bigcap_{i=1}^{n} A_{i}^c)=0$.
By Proposition \ref{p2.9} we have that  
$
    V^v\le_{st} \widetilde W^{v,h},
$
where $h(v):=i$ for $v\in A_i$ and $h(v):=1$, otherwise.
Therefore, the required result follows from Theorem \ref{t:subcrit_ex}.
\end{proof}

\begin{example}\label{ex:hyperbolic}\rm Take $\BX$ as the $d$-dimensional hyperbolic space $\BH^d$ for some $d\ge 2$
equipped with the hyperbolic metric $d_{\BH^d}$; ; see e.g.\ \cite{HLS24} and the references given there. 
Assume that $\lambda$ is given by the Haar measure $\mathcal{H}^d$ on $\R^d$.
Assume that the connection function is given by $\varphi(x,y)=\tilde{\varphi}(d_{\BH^d}(x,y))$ for
some measurable $\tilde\varphi\colon\R_+\to[0,1]$. 
Fix a point $o\in \BH^d$.
Since the space $\BH^d$ is homogeneous, we can argue as
in Remark \ref{r:432} to see that
\begin{align*}
t_c=\sup\{t\ge 0: \BP_t(|C^o|<\infty)=0),\quad
t_T=\sup\big\{t\ge 0: \BE_t |C^o|<\infty\big\}.
\end{align*}
It was proved in \cite{DicksonHeydenreich25} that 
$t_c<\infty$ if and only if $\int\varphi(o,x)\,\mathcal{H}^d(dx)>0$. In accordance with our 
Lemma \ref{l:not_triv} it was also shown there that
$t_T>0$ if and only if $\int\varphi(o,x)\,\mathcal{H}^d(dx)<\infty$.
Assume now that $\int\varphi(0,x)\,\mathcal{H}^d(dx)\in(0,\infty)$. It was proved in \cite{DicksonHeydenreich25} 
that $t_c=t_T$. Hence our 
Theorem \ref{tmoments_2} yields a strong sharp phase transition at $t_c$, just as in the stationary (unmarked) Euclidean case.
\end{example}

\section{Diameter distribution}\label{secDiameter}

We establish the setting of Section \ref{secMarkov} with an intensity measure $t\lambda$ for $t\ge 0$. 
Denote by 
\begin{align*}
    d_m(A_1, A_2):=\sup_{x\in A_1, y\in A_2} d(x,y)
\end{align*}
the maximum distance between the points of $A_1,A_2\subset\BX$, 
where we recall that $d(\cdot,\cdot)$ denotes the metric on $\BX$. We use
the convention that $d_m(A, \emptyset):=0$ for any
$A\subset\BX$.  We will also use the same notations for $\mu\in\bG$ identifying $A$ with $V(\mu)$.

\begin{lemma}\label{l:subadditive_D}
    Let $A_1,A_2\subset \BX$, $A=A_1 \cup A_2$, and $v\in A$. If $A_1\cap A_2\ne\emptyset$, then
    \begin{align*}
        D(A)\le \min(D(A_1)+D(A_2), 2 d_m(v,A)).
    \end{align*}
\end{lemma}
\begin{proof}
    Note that for any $x\in A_1\cap A_2$ by the triangle inequality, we have 
    \begin{align*}
        d_m(A_1,A_2)\le d_m(x,A_1)+d_m(x,A_2)\le d_m(A_1\cap A_2,A_1)+d_m(A_1\cap A_2,A_2)\le D(A_1)+D(A_2).
    \end{align*}
    Therefore, the result follows from the fact that $D(A)=\max(D(A_1),D(A_2),d_m(A_1,A_2))$, and a simple consequence of the triangle inequality, i.e. $D(A)\le 2 d_m(v,A)$.
\end{proof}

Denote for $v\in\BX$ and $n\in\N$ the maximum length of the edges between generations $n-1$ and $n$ by
\begin{align*}
    E^v_n:=\max_{x\in V^v_{n-1}, y\in V^v_n : x\sim y} d(x,y).
\end{align*}
We also define by $E^v_{\le n}:=\max(E^v_1,E^v_2,\dots,E^v_n)$ and $E^v:=\sup_{n\in\N} E^v_{\le n}$ the maximal length among the edges between generations in $C^v_{\le n}$ and $C^v$ respectively. Note that $E^v_1 = d_m(v,V^v_1)$. For $v\in\BX$ and $t, r\ge 0$ we denote by 
\begin{align*}
     \phi^v_t(r):=1-\exp\left( - t \int_{B^c_r(v)} \varphi(v,x) \, \lambda(d x)\right),
\end{align*}
where $B_r(v):=\{x\in\BX : d(v,x)\le r \}$ is the closed ball of radius $r$ centered at a point $v$ in $\BX$. 
We also denote by $\phi^*_t(r):=\esssup_{v\in\BX}\phi^v_t(r)$.

\begin{lemma}
    Let $v\in\BX$ and $t, r\ge 0$. Then $\BP_t(E^v_1>r)=\phi^v_t(r)\le \phi^*_t(r)$.
\end{lemma}
\begin{proof}
    The claim follows from Proposition \ref{p2.1} since $V^v_1$ is a Poisson process with intensity measure $K_{t\lambda}(0,\delta_v)$.  
\end{proof}

\begin{proposition}\label{p:edge_moments}
    Let $v\in\BX$, $n\in\N$ and $r\ge 0$. Then  
    \begin{align*}
        \BP_t(E^v_{\le n}>r)\le c^v_{\le n-1}(t) \phi^*_t(r), \quad  \BP_t(E^v>r)\le c^v(t) \phi^*_t(r).
    \end{align*}
\end{proposition}
\begin{proof}
    By subadditivity of a probability measure and Proposition \ref{p2.1}, we have
    \begin{align*}
        \BP_t(E^v_{\le n}>r)=\BP_t\left(\bigcup_{k=1}^n\{E^v_{k}>r\}\right)
&\le \sum_{k=1}^n \BP_t(E^v_{k}>r)
       =\sum_{k=1}^n \BP_t\left(\bigcup_{x\in V^v_{k-1}}\{\max_{y\in V_k^v : x \sim y} d(x,y) >r \} \right) \\  
&\le \sum_{k=1}^n \BE_t \int \BP_t\left( \max_{y\in V_k^v : x \sim y} d(x,y) >r \mid C^v_{\le k-1}\right) \, V^v_{k-1}(d x) \\ 
&=\sum_{k=1}^n \BE_t \int 1-\exp\left( - t \int_{B^c_r(x)} \varphi(x,y) \bar\varphi(V^v_{\le k-2},y) \, \lambda(d y)\right) \, V^v_{k-1}(d x) \\ 
&\le \sum_{k=1}^n \BE_t \int \phi^x_t(r) \, V^v_{k-1}(d x) \le c^v_{\le n-1}(t) \phi^*_t(r).
    \end{align*}
The second inequality follows immediately from the first one, since $E^v_{\le n}\uparrow E^v$ a.s. as $n\to\infty$.
\end{proof}

Denote by $\tau^v:=\min( n\in\N : V^v_n=\emptyset)$ the depth of $C^v$.

\begin{lemma}\label{l:diam_edge_bound} 
Let $v\in\BX$ and $n\in\N$. Then with probability one 
    \begin{align*}
        d_m(v,V^v_{\le n})\le n E^v_{\le n}, \quad   d_m(v,V^v)\le \tau^v E^v.
    \end{align*}
\end{lemma}
\begin{proof}
    Denote for $\mu\in\bG$ and $v_1,v_2\in\mu$ the minimal length of the paths between $v_1$ and $v_2$ within the graph $\mu$ by $d^\mu(v_1,v_2)$. Note that for any $k\in\N$ by the triangle inequality, we have 
    \begin{align*}
        d_m(v,V^v_{k})= \max_{x\in V^v_k} d(v,x)\le \max_{x\in V^v_k} d^{C^v_{\le k}}(v,x)\le E_1^v+\dots+E_k^v.
    \end{align*}
    Therefore
    \begin{align*}
        d_m(v,V^v_{\le n})=\max(d_m(v,V^v_{1}),d_m(v,V^v_{2}), \dots , d_m(v,V^v_{n}))\le E_1^v+\dots+E_n^v\le n E^v_{\le n}.
    \end{align*}
    The second inequality follows from the first one by monotone convergence. 
\end{proof}

\begin{proposition}\label{p:diameter_moments}
    Let $v\in\BX$, $n\in\N$ and $t, r\ge 0$. 
    Then
    \begin{align*}
        \BP_t(d_m(v,V^v_{\le n})>r)\le c^v_{\le n-1}(t) \phi^*_t(r/n),
        \quad \BP_t(d_m(v,V^v)>r)\le c^v_{\le n-1}(t)\phi^*_t(r/n) + \BP_t(\tau^v>n).
    \end{align*}
\end{proposition}
\begin{proof}
By Lemma \ref{l:diam_edge_bound} and Proposition \ref{p:edge_moments} we have
    \begin{align*}
         \BP_t(d_m(v,V^v_{\le n})>r)\le  \BP_t(n E^v_{\le n}>r)\le c^v_{\le n-1}(t)\phi^*_t(r/n).
    \end{align*}
   The second inequality follows from the first one and the following inclusion
   \begin{align*}
       \{d_m(v,V^v)>r\}\subset \{d_m(v,V^v_{\le n})>r\} \cup \{ \tau^v > n \}.
   \end{align*}
\end{proof}

\begin{corollary}\label{c:diameter_moments}
     Let $v\in\BX$, $n\in\N$ and $t, r\ge 0$. 
    Then
    \begin{align*}
        \BP_t(D(V^v_{\le n})>r)\le c^v_{\le n-1}(t)\phi^*_t(r/2n), 
        \quad \BP_t(D(V^v)>r)\le c^v_{\le n-1}(t)\phi^*_t(r/2n) + \BP_t(\tau^v>n).
    \end{align*}
\end{corollary}
\begin{proof}
    The claims follow immediately from Lemma \ref{l:subadditive_D} and Proposition \ref{p:diameter_moments}.
\end{proof}

\begin{corollary}\label{c:diameter_ex}
    Let $n\in\N$ and $t\ge 0$. If $D^*_\varphi<\infty$ and $\phi^*_t(r)$ decay exponentially fast as $r\to\infty$, then
    there exists $\delta:=\delta(t,n)>0$ such that 
    \begin{align*}
     \esssup_{v\in\BX} \BE_t e^{\delta D(V^v_{\le n})}<\infty.
     \end{align*}
\end{corollary}
\begin{proof}
    The claim follows from Corollaries \ref{c:diameter_moments} and \ref{cmoments}.
\end{proof}

\begin{remark}\rm
  Suppose that $\phi^*_t(r)$ and the tail distribution of depth $\BP_t(\tau^v>r)$ decay exponentially fast as $r\to\infty$ with the
  exponents $\delta_1>0$ and $\delta_2>0$ respectively. Then using
  Proposition \ref{p:diameter_moments} one can easily show that the tail  distribution of $d_m(v,V^v)$ decrease as  $\exp(-\sqrt{\delta_1 \delta_2 r})$ via $r$. In other words, using
  the previous Proposition, we can't prove the light tail property for
  the diameter of a cluster. We will use a similar construction as in Theorem \ref{t:subcrit_ex} to achieve the desired result; see
  Theorem \ref{t:diameter_ex}. For example, if $\phi^*_t(r)$ has a heavy-tail (e.g., decay with
  polynomial speed) and depth $\tau^v$ has a light-tailed distribution, then one can build a "better" upper bound for the tail distribution of $d_m(v,V^v)$ using Proposition \ref{p:diameter_moments}.     
\end{remark}

\begin{theorem}\label{t:diameter_ex}
Let $h\colon \BX\to\N$ be a measurable and bounded function and $t\ge 0$. 
Assume that $D^*_\varphi<\infty$ and $c^*_h(t):=\esssup\limits_{v\in\BX} c^v_{h(v)}(t)<1$.
Assume also that $\phi^*_t(r)$ decays exponentially fast as $r\to\infty$, 
then there exists $\delta\equiv\delta(t, h^*)>0$ such that 
\begin{align*}
\esssup_{v\in\BX} \BE_t e^{\delta D(\widetilde W^{v,h})}<\infty.
\end{align*}
\end{theorem} 
\begin{proof}
To describe $W^{v,h}$ and $\widetilde W^{v,h}$ we will use the same 
exploration process (spatial Markov process) $(X_k)_{k\ge 0}=((Y_k,Z_k))_{k\ge 0}$ 
as in Theorem \ref{t:subcrit_ex}. Recall that $X_0=(\delta_v,\delta_v)$, $X_1=(V^{v}_{\le h(v)},V^{v}_{ h(v)})$ and
\begin{align*}
    Y_{k+1}=&  Y_k + \int \I\{d(v,x)=\min_{y\in Z_k} d(v,y)\} V^{x!}_{\le h(x)}(\xi_x^k) \, Z_k(d x), \\
    Z_{k+1}=& Z_k+\int \I\{d(v,x)=\min_{y\in Z_k} d(v,y)\} (V^x_{h(x)}(\xi_x^k)-\delta_x) \, Z_k(d x).
\end{align*}
Let $\tau:=\min\{k\ge 1 : Z_k=\emptyset\}$. We know that $\tau$ is a stopping time 
w.r.t.\ $\{Z_k\}_{k\in\N}$, $\tau=|W^{v,h}|$ and $Y_\tau=\widetilde W^{v,h}$ in distribution.

Let $\delta$ be a positive and sufficiently small parameter. By Lemma \ref{l:subadditive_D} we have for any $k\ge 1$
\begin{align*}
    \BE_t e^{\delta D(Y_k)}&\le\BE_t \left(e^{\delta D(Y_{k-1})} 
\BE_t \left(\exp\left(\delta \int \I\{d(v,x)=\min_{y\in Z_{k-1}} d(v,y)\} D(V^x_{\le h(x)}(\xi_x^{k-1}))\, Z_{k-1}(d x)\right)\Bigm| X_{k-1}\right)\right) \\
&=\BE_t \left(e^{\delta D(Y_{k-1})} \int \I\{d(v,x)
=\min_{y\in Z_{k-1}} d(v,y)\} \BE_t (e^{\delta D(V^x_{\le h(x)}(\xi^{k-1}_x))}\Bigm| Z_{k-1})\, Z_{k-1}(d x)\right) \\
&\le \BE_t e^{\delta D(Y_{k-1})} \esssup_{v\in\BX}\BE_t e^{ \delta  D(V^v_{\le  h^*})} 
\le \left(\esssup_{v\in\BX}\BE_t e^{ \delta  D(V^v_{\le  h^*})}\right)^k.
\end{align*}
By Theorem \ref{t:subcrit} $|\widetilde W^{v,h}|$ is a proper random variable. 
Then, by Fubini's theorem, Cauchy–Schwarz inequality, and the previous inequality, we also have 
\begin{align*}
   \BE_t e^{ \delta D\left( \widetilde W^{v,h}\right) }
&=\BE_t \left( e^{ \delta D\left( \widetilde W^{v,h}\right) }\sum_{k=1}^\infty \I\{\tau=k\}\right)=\sum_{k=1}^\infty \BE_t (e^{\delta D(Y_k)} \I\{\tau=k\})\\
&\le \sum_{k=1}^\infty \sqrt{\BE_t e^{2\delta D(Y_k)} \BP_t(\tau=k)}
    \le \sum_{k=1}^\infty \sqrt{\big(\esssup_{v\in\BX}\BE_t e^{ 2 \delta  D(V^v_{\le  h^*})}\big)^k \BP_t(\tau=k)}.
\end{align*}
Note that $D(V^v_{\le  h^*})$ and $\tau$ have light tail
distributions uniformly in $v\in\BX$ by Corollary \ref{c:diameter_ex}
and Theorem \ref{t:subcrit_ex} respectively. Therefore, we get the
required result, since
$\lim\limits_{\delta\to 0}\esssup_{v\in\BX} \BE_t e^{2\delta
  D(V^v_{\le h^*})}=1$.
\end{proof}

\begin{corollary}\label{c:subcrit_diam_ex}
Suppose that $\phi^*_t(r)$ decays 
exponentially as $r\to\infty$. Then, under the conditions of Corollary \ref{c:subcrit}, there exists $\delta=\delta(t,n)>0$ such that
    \begin{align*}
       \esssup_{v\in\BX} \BE_t e^{\delta D(V^{v})}<\infty.
    \end{align*}
\end{corollary}
\begin{proof}
The assertion follows from Corollary \ref{p2.8} and Theorem \ref{t:diameter_ex} with $h\equiv n$. 
\end{proof}

\begin{theorem}\label{t_diametr_moments_2}
Let $t<t_T$. If $\phi^*_t(r)$ decays exponentially as $r\to\infty$, 
then the diameter of the cluster of an arbitrary vertex has a light-tailed distribution (as in Corollary \ref{c:subcrit_diam_ex}).
\end{theorem}
\begin{proof}
We can build the same upper bound for the cluster of an arbitrary vertex as in Theorem \ref{tmoments_2} and then 
derive the required result with Theorem \ref{t:diameter_ex}.
\end{proof}

\section{The stationary marked RCM}\label{s:stationaryRCM}

In this section, we consider the stationary RCM as introduced in Section \ref{s:perc}.
Let $n\in\N$ and define the measurable functions
$d_\varphi^{(n)}\colon \BM^2\to[0,\infty]$ and
$d_\varphi^{[n]}\colon \BM^2\to[0,\infty]$ by
\begin{align}\notag
d_\varphi^{(n)}(p,q):=&\iint \prod_{i=0}^{n-1} \varphi((x_i,p_i),(x_{i+1},p_{i+1})) \,d \boldsymbol{x}_n \, \BQ^{n-1}(d \boldsymbol{p}_{n-1}), \\
\label{e:653}
d_\varphi^{[n]}(p,q):=&\iint \prod_{i=0}^{n-1} \varphi((x_i,p_i),(x_{i+1},p_{i+1})) 
\prod_{3\le i+2\le j\le n} \bar\varphi ((x_i,p_i), (x_{j},p_j))
\,d \boldsymbol{x}_n \, \BQ^{n-1}(d \boldsymbol{p}_{n-1}).
\end{align}
where $p_0:=p$ and $p_n:=q$. In the special case $n=1$, we write
\begin{align*}
d_\varphi(p,q):=d_\varphi^{(1)}(p,q)=\int \varphi(x,p,q)\,dx.
\end{align*}
Note that
\begin{align*}
    d_\varphi^{(n)}(p,q):=\int \prod_{i=0}^{n-1} d_\varphi(p_i,p_{i+1}) \,\BQ^{n-1}(d \boldsymbol{p}_{n-1}).
\end{align*}
It follows from the Mecke equation  \eqref{e:Meckev} that $t^n d_\varphi^{(n)}(p,\cdot)$ is the $\BQ$-density
of the expected number of paths of length $n$ starting in $(0,p)$ and ending in a point with mark in a given
set from $\cB(\BM)$. Analogously, $t^n d_\varphi^{[n]}(p,\cdot)$
is the $\BQ$-density of the expected number of {\it self-avoiding walks} 
(paths without loops) of length $n$ starting in $(0,p)$ and ending in a point with a mark 
in a set from $\cB(\BM)$.
From the symmetry property of $\varphi$ we obtain 
that $d_\varphi^{(n)}$ and $d_\varphi^{[n]}$ are symmetric. From stationarity, one can also notice that for given 
$n\ge 1$ and $p,q\in\BM$ we have 
\begin{align}\label{d_[n]_submult}
    d_\varphi^{[2n]}(p,q)\le \int d_\varphi^{[n]}(p,r) d_\varphi^{[n]}(r,q)\, \BQ(dr).
\end{align}
For a given measurable $L\colon \BM^2\to\R\cup\left\{\pm\infty\right\}$ and $r_1,r_2\in\left[1,\infty\right)$, we define the norms
\begin{align}
    \|L\|_{r_1,r_2} :=& \left(\int \left(\int |L(p,q)|^{r_1}\, \BQ(d p)\right)^\frac{r_2}{r_1}\, \BQ(d q)\right)^\frac{1}{r_2},\\
    \|L\|_{\infty,r_2} :=& \esssup_{p\in\BM}\left(\int |L(p,q)|^{r_2}\, \BQ(d q)\right)^\frac{1}{r_2},\\
    \|L\|_{\infty,\infty} :=& \esssup_{p,q\in\BM} |L(p,q)|.
\end{align}

Note that, by definition, we have $D^*_\varphi=\|d_\varphi\|_{\infty,1}$.
Take $p\in\BM$, $n\in\N$ and $t\ge 0$. By the multivariate Mecke equation, we have that
\begin{align}\label{e:vpq}
\BE_tV^{(0,p)}_n(\R^d\times\cdot)=\int \I\{q\in\cdot\}v^{p,q}_n(t)\,\BQ(dq),
\end{align}
where the density $v^{p,q}_n(t)$ can be written as a linear combination of
integrals similar to those occurring in  \eqref{e:653}.
Since we can bound $V^{(0,p)}_n(\R^d\times B)$ by counting all
paths of length $n$ without loops and ending in a measurable $B\subset\BM$, 
we have
\begin{align}\label{cnbound1}
 v^{p,q}_n(t)\le t^n d_\varphi^{[n]}(p,q) \le t^n d_\varphi^{(n)}(p,q),\quad \BQ\text{-a.e.\ $q\in\BM$}
\end{align}
and therefore
\begin{align}\label{cnbound2}
 v^p_n(t):=v^{(0,p)}_n(t)\le t^n\int d_\varphi^{[n]}(p,q)\,\BQ(dq)
\le t^n\int d_\varphi^{(n)}(p,q)\,\BQ(dq).
\end{align}
We define 
\begin{align}\label{e:Delta_n}
    \Delta_n(t):=c^*_{\le n-1}(t) +\|v_n(t)\|_{\infty,\infty},
\end{align}
where we recall that the first term is the $\|\cdot\|_\infty$-norm of $p\mapsto  c^{(0,p)}_{\le n-1}(t)$
and the second term is defined as the $\|\cdot\|_{\infty,\infty}$-norm of the function
$(p,q)\mapsto  v^{p,q}_n(t)$.

\begin{lemma}\label{l:Delta_n}
Let $t\ge 0$ and $n\in\N$. If $\|d_\varphi\|_{\infty,1}<\infty$ and $\|d_\varphi^{[n]}\|_{\infty,\infty}<\infty$, 
then $\Delta_{n}(t)<\infty$. Moreover, if  $\|d_\varphi\|_{\infty,1}<\infty$ and $\|d_\varphi^{[n]}\|_{\infty,2}<\infty$, then $\Delta_{2n}(t)<\infty$.
\end{lemma}
\begin{proof}
By the recursive structure of $d^{(n)}$ we have $\|d_\varphi^{(n)}\|_{\infty,1}\le \|d_\varphi\|_{\infty,1}^n$. 
Therefore, we obtain from \eqref{cnbound2} that
\begin{align*}
c^*_{\le 2n-1}(t)\le \sum_{k=0}^{2n-1} \|v_k(t)\|_{\infty}
\le 1+ \sum_{k=1}^{2n-1} t^k\|d_\varphi^{(k)}\|_{\infty,1}\le \frac{(t\|d_\varphi\|_{\infty,1})^{2n}-1}{t\|d_\varphi\|_{\infty,1}-1},
 \end{align*}
where the final upper bound has to be interpreted as $2n$ if $t\|d_\varphi\|_{\infty,1}=1$. 
On the other hand, we obtain from the inequalities \eqref{cnbound1}, \eqref{d_[n]_submult} 
and the Cauchy--Schwarz inequality, 
\begin{align*}
\|v_{2n}(t)\|_{\infty,\infty} \le t^{2n} \|d_\varphi^{[2n]}\|_{\infty,\infty}
\le t^{2n}\esssup_{p,q\in\BM} \int d_\varphi^{[n]} (p,r) d_\varphi^{[n]} (r,q) \, \BQ(d r) \le t^{2n} \| d_\varphi^{[n]}\|_{\infty,2}^2.
\end{align*}
Hence
\begin{align}\label{e:Delta_2n}
\Delta_{2n}(t)\le \frac{(t\|d_\varphi\|_{\infty,1})^{2n}-1}{t\|d_\varphi\|_{\infty,1}-1}+  t^{2n} \| d_\varphi^{[2n]}\|_{\infty,\infty}
\le \frac{(t\|d_\varphi\|_{\infty,1})^{2n}-1}{t\|d_\varphi\|_{\infty,1}-1}+  t^{2n} \| d_\varphi^{[n]}\|_{\infty,2}^2.
\end{align} 
\end{proof}

Finally, we are ready to state our main result on the strong sharpness of the phase transition, which is a significant generalization of the main result from \cite{Ziesche18} and some of the results from \cite{CaicedoDickson24}. The main condition under which we can prove the strong sharpness is that there exists $n\in\N$ satisfying
\begin{align}\label{c:main}
    \|d_\varphi\|_{\infty,1}+ \|d_\varphi^{[n]}\|_{\infty,\infty}<\infty.
\end{align}

\begin{theorem}\label{t:main} Assume that $\|d_\varphi\|_{1,1}>0$.  We have the following:
\begin{enumerate}
\item[{\rm (i)}] $t_T\ge \|d_\varphi\|_{\infty,1}^{-1}$.
        \item[{\rm (ii)}] For $t<t_T$ there exists $\delta_1:=\delta_1(t)$ 
such that $\esssup\limits_{p\in\BM}\BE_t e^{\delta_1 |V^{(0,p)}|}<\infty$.
        \item[{\rm (iii)}]  If 
        \begin{align*}
          \esssup\limits_{p\in\BM} \int_{\|x\|>u} \varphi(x,p,q) \, d x \, \BQ(d q)
        \end{align*}
        decays exponentially fast as $u\to\infty$, then for $t<t_T$ there exists $\delta_2:=\delta_2(t)$ such that 
        \begin{align*}
            \esssup\limits_{p\in\BM}\BE_t e^{\delta_2 D(V^{(0,p)})}<\infty.
        \end{align*}
       \item [{\rm (iv)}] Suppose $n\in\N$ satisfies \eqref{c:main}. Then $t_T=t_c\in(0,\infty)$ and 
$\bar{c}_{t_c}=\BE_{t_c}\int |V^{(0,p)}|\,\BQ(dp)=\infty$.
        \item[{\rm (v)}] Suppose $n\in\N$ satisfies \eqref{c:main}. Then
        \begin{align*}
        \|c(t)\|_r\ge \frac{t_c}{\Delta_n(t)(t_c-t)},
    \end{align*}
    for $t<t_c$ and for all $r\in[1,\infty]$. 
       \item[{\rm (vi)}] Suppose that $n\in\N$ satisfies \eqref{c:main} and let $\delta>0$.  
Then 
\begin{align*}
    \|\theta(t)\|_r \geq \left(\frac{\bar\theta(t_c)}{\delta}+\frac{\I\{\bar\theta(t_c)=0\}}{2t\Delta_n(t)}\right)(t-t_c),
\end{align*}
for all $t\in[t_c, t_c + \delta]$ and $r\in\left[1,\infty\right]$. 
\end{enumerate}
\end{theorem}
\begin{proof} Assertion (i) follows from Lemma
  \ref{l:not_triv}, (ii) follows from Theorem \ref{tmoments_2},
 (iii) follows from Theorem \ref{t_diametr_moments_2}. Other claims will be proven later on in this section. More precisely, (iv)
  follows from Proposition \ref{p:equality_t_T}, Corollary
  \ref{c:sharpness_of_ph_tran} and the definition of $t^{(1)}_T=t_c$,
(v) follows from Proposition \ref{p:susc_mean_field_bound} and
(vi)  follows from Theorem \ref{t:per_mean_field_bound}.
\end{proof}

\begin{remark}\rm For $n=1$ the condition \eqref{c:main} boils down to
$\|d_\varphi\|_{\infty,\infty}<\infty$. This is the main assumption made in \cite{CaicedoDickson24}. 
\end{remark}

\begin{remark}\rm \label{r:equality_t_T_inf}
Let $k\in\N$. By \eqref{e:Delta_2n} the condition 
   \begin{align}\label{e:equality_t_T_inf}
        \|d_\varphi\|_{\infty,1}+ \|d_\varphi^{[k]}\|_{\infty,2}<\infty
\end{align}
is sufficient for \eqref{c:main} with $n=2k$.
\end{remark}

\begin{remark}\rm \label{r:sharpness}
  Note that if $\phi^*_t(u)$ has a heavy tail in $u$ (thicker than exponential), then the diameter of the cluster of an arbitrary vertex has a heavy tail distribution (thicker than $\phi^*_t(u)$), which can be shown by Proposition \ref{p:diameter_moments}.
\end{remark}

\subsection{Sufficient condition for non-triviality of the phase transition}

We start with the following simple observation. 

\begin{proposition}\label{p:nottriv_t_T}
    We have that $t^{(1)}_T\ge \|d_\varphi\|_{2,2}^{-1}$.
\end{proposition}
\begin{proof}
    The claim is trivial for $\|d_\varphi\|_{2,2}\in \{0,\infty\}$. Suppose that $\|d_\varphi\|_{2,2}\in (0,\infty)$.
    Let $n\in\N$. From the Cauchy–Schwarz inequality we have
    \begin{align*}
\|d^{(n)}_\varphi\|_{1,2}^2 = \int \left( \int d^{(n-1)}_\varphi (p,q_1) d_\varphi(q_1,q) \, \BQ^2 (d(p,q_1)) \right)^2 \, \BQ(d q) \le  
\|d_\varphi\|_{2,2}^2 \|d^{(n-1)}_\varphi\|_{1,2}^2.
    \end{align*}
Therefore $\|d^{(n)}_\varphi\|_{1,2}\le \|d_\varphi\|_{2,2}^n$.
On the other hand, we obtain from \eqref{cnbound2} and Jensen's inequality that
    \begin{align*}
        \bar{c}_n(t)\le t^n \|d^{(n)}_\varphi\|_{1,1}\le t^n \|d^{(n)}_\varphi\|_{1,2},\quad t>0.
    \end{align*}
Therefore,
\begin{align*}
\bar{c}(t)=\sum_{n=0}^\infty \bar{c}_n(t)\le \sum_{n=0}^\infty t^n \|d_\varphi\|_{2,2}^n, 
\end{align*}
which converges if $t<\|d_\varphi\|_{2,2}^{-1}$.      
\end{proof}

\begin{remark}\label{r:Gilbert}\rm
Consider the Gilbert graph from Example \ref{ex:Gilbert} and assume that $\BQ\{0\}<1$. 
It was proved in
\cite{Hall85} that $t^{(1)}_T\in(0,\infty)$ if and only if $q_2:=\int r^{2d}\,\BQ(dr)<\infty$.
Note that $q_2<\infty$ is equivalent to $\|d_\varphi\|_{2,2}>0$.
If $q_2=\infty$ then it was shown in \cite{Hall85} that 
$c^{(0,p)}(t)=\infty$ for all $p\ge 0$ and $t>0$, so that $t_T=t^{(1)}_T=0$ in this case.
The authors of \cite{GouTheret19} introduce another
critical intensity $\hat{t}\le t_c$ (called $\hat{\lambda}_c$ on p.\ 3717). If $q_2<\infty$, then
Theorem 2 in this paper shows that $\BE_t\lambda_d(Z_0)<\infty$
for all $t<\hat{t}$, where $Z_0$ denotes the union of all
balls $B(x,r)$, $(x,r)\in\eta$, with $0\in B(x,r)$. 
Moreover, we then also have
\begin{align*}
\BE_t\eta(Z_0\times[0,\infty))<\infty,\quad t<\hat{t},
\end{align*}
Later it was proved in \cite{Duminiletal2020} that $t_c=\hat{t}$ provided
that $\int r^{5d-3}\,\BQ(dr)<\infty$.
\end{remark}

We continue with the following simple fact.

\begin{proposition}\label{p:nottriv_t_c}
Suppose $A\in \cB(\BM)$ satisfies $\BQ(A)>0$. Assume that for some
symmetric function $\varphi_0\colon \R^d\to [0,1]$, such that $\int \varphi_0(x)\, dx \in(0,\infty)$, we have   
\begin{align*}
    \varphi(x,p,q)\ge \varphi_0(x), \ \lambda_d\otimes\BQ^2\text{-a.e. } (x,p,q)\in\R^d\times A \times A.
\end{align*}
Then $t_c<\infty$.
\end{proposition}
\begin{proof}  Let $\xi'$ be a RCM driven by $\eta_{\R^d\times A}$ with a connection
function $\varphi_0$. It is easy to build a coupling such that
$\xi'\subset \xi[\R^d\times A]$. By \cite{Pen91} we know that $\xi'$
percolates for large intensity. Therefore, $t_c<\infty$ for  $\xi$ as well.
\end{proof}

\begin{proposition}\label{p:nottriv_t_c_2}
    Assume that $\|d_\varphi\|_{1,1}\in(0,\infty)$, then $t_c<\infty$.
\end{proposition}
\begin{proof}
  Suppose that the mark space contains only two marks
  $\BM=\{p_1, p_2\}$, and a mark cannot directly connect with
itself, i.e.\ $\int\varphi(x,p_i,p_i)\,dx=0$ for $i\in\{1,2\}$
  (otherwise we may refer to Proposition \ref{p:nottriv_t_c}). 
In this case, we can show $t_c<\infty$ as in \cite[Lemma  2.2]{CaicedoDickson24}, 
whose proof extends the approach from  \cite{Pen91} to the marked case.
Note that non-triviality of $\|d_\varphi\|_{1,1}$
implies existence of $\varepsilon>0$, $C\in\cB(\R^d)$ and
  $A,B\in \cB(\BM)$ with $\min\{\lambda_d(C),\BQ(A),\BQ(B)\}>0$ such that
    \begin{align*}
    \varphi(x,p,q)=\varphi(-x,q,p)\ge \varphi_0(x), \ \lambda_d\otimes\BQ^2\text{-a.e. } (x,p,q)\in\R^d\times A \times B,
    \end{align*}
where $ \varphi_0(x):=\varepsilon\I\{x\in\pm C\}$. We can again show $t_c<\infty$ by an appropriate coupling
($\xi'\subset \xi[\R^d\times (A\cup B)]$) with the connection function $\varphi_0(x)$ as in Proposition \ref{p:nottriv_t_c}.
\end{proof}

\begin{remark}\label{r:triv_t_T_inf}\rm
By definition, we have $D^*_\varphi=\|d_\varphi\|_{\infty,1}$. Therefore, we obtain from Lemma \ref{l:not_triv} that $\|d_\varphi\|_{\infty,1}<\infty$ is equivalent to $t_T>0$ and hence implies $t_c>0$. On the other hand,
$\|d_\varphi\|_{\infty,1}<\infty$ implies $\|d_\varphi\|_{1,1}<\infty$. Hence, if
$0<\|d_\varphi\|_{\infty,1}<\infty$ then Proposition \ref{p:nottriv_t_c_2} shows that $t_c<\infty$ and hence also $t_T<\infty$.
Altogether we see that $0<\|d_\varphi\|_{\infty,1}<\infty$ 
is necessary and sufficient for  $t_T\in(0,\infty)$,  
and sufficient for $t_c\in(0,\infty)$. 
\end{remark}

\begin{remark}\rm \label{r:suff}  
Notice that 
    \begin{align*}
\| d_\varphi\|_{1,1} \le \max(\| d_\varphi\|_{\infty,1},\| d_\varphi\|_{2,2})
\le \| d_\varphi\|_{\infty,2}\le \| d_\varphi\|_{\infty,\infty}.
    \end{align*}
Therefore, by Remark \ref{r:equality_t_T_inf} the condition  $\| d_\varphi\|_{\infty,2}<\infty$ 
is sufficient for \eqref{c:main} with $n=2$.
\end{remark}

\begin{example}\label{ex:min-reach}\rm
Let $\xi$ be a stationary marked RCM 
and suppose that $R\colon \BM^2\to \R_+$ is a  symmetric and meassurable
function   such that $\| R^{d}\|_{\infty,2}<\infty$. Assume that
\begin{align*}
    \varphi(x,p,q)\le \I\{ |x|\le R(p,q)\}, \ \lambda_d\otimes\BQ^2\text{-a.e. } (x,p,q)\in\R^d\times \BM^2.
\end{align*}
Then $d_\varphi(p,q)\le \kappa_d R^d(p,q)$ for $\BQ^2$-a.e. $(p,q)\in\BM^2$, 
where $\kappa_d$ is the volume of a unit ball in $\R^d$. Therefore, $\xi$ undergoes a 
strong sharp phase transition, since $\| d_\varphi\|_{\infty,2}\le \kappa_d\| R^{d}\|_{\infty,2}<\infty$. 

For example, let $\xi$ be a {\em min-reach} RCM (see
\cite{CaicedoKolesnikov25}) with $\BM=\R_+$ and
$R(p,q)=R_0(\min(p,q))$, where $R_0\colon \R_+\to \R_+$ is a
non-decreasing function. Assume that the {\em reach function} $R_0$ has a
finite moment of order $2d$ with respect to the mark distribution
$\BQ$. Then 
\begin{align*}
    \| R^{d}_0\|^2_{\infty,2}
\le \int R^{2d}_0(q)\,\BQ( dq)<\infty.
\end{align*}
Hence $\xi$ has a strong sharp phase transition.
In the special case $\varphi(x,p,q)=\I\{|x|\le \min(p,q)\}$, all the randomness comes from
the stationary marked Poisson process. The RCM $\xi$ is then a random version of the {\em symmetric random disk graph}; see \cite{Abuetal12}. It has a strong sharp phase transition if the radius distribution $\BQ$ has a finite moment of order $2d$.  
\end{example}

\begin{example}\label{ex:weighted_2}\rm
 Let $\varepsilon>0$ and consider
 Example \ref{ex:weighted} with
 $g(p,q):=(p\vee q)^\varepsilon:=\max(p,q)^\varepsilon$.  All that follows also applies to $g(p,q):=(p+q)^\varepsilon$ 
as well, while $p+q\ge p\vee q$.
It is easy to see that  $\|d_\varphi\|_{\infty,1}<\infty$ if and only if $\varepsilon<1$.
Assume $\varepsilon<1$ and take $n\in\N$. We have
\begin{align*}
|d_\varphi^{(n)}\|_{\infty,2}^2 \le 
m_\rho^{2n} \int_0^1 \left(\int_0^1 [p_1(p_1\vee p_2)\cdots (p_{n-1}\vee p_n)]^{-\varepsilon} 
\, d \boldsymbol{p}_{n-1} \right)^2 \, d p_n.
\end{align*}        
For $\varepsilon<1/2$ we obtain
$\|d_\varphi\|_{\infty,2}<\infty$ from a direct calculation.  
For $\varepsilon\in [1/2,1)$, it is also not difficult to show that 
$\|d_\varphi^{(n)}\|_{\infty,2}<\infty$ if
$\varepsilon\in[1-\frac1{2(n-1)},1-\frac1{2n})$ for some $n\ge 2$.
Therefore $\|d_\varphi^{(2)}\|_{\infty,2}<\infty$ for $\varepsilon=1/2$
while for $\varepsilon>1/2$ we have
$\|d_\varphi^{(n)}\|_{\infty,2}<\infty$,  where $n=\lceil(2-2\varepsilon)^{-1}\rceil$.
Altogether we obtain for each $\varepsilon\in(0,1)$ that
our condition  \eqref{e:equality_t_T_inf} holds. Hence, we have a strong sharp phase transition at $t_c=t_T$.

We note in passing that the condition
$\|d_\varphi\|_{2,2}<\infty$ is also equivalent with $\varepsilon<1$, while
condition $\|d_\varphi\|_{1,1}<\infty$ is equivalent
with $\varepsilon<2$.
\end{example}

\begin{remark}\rm  In the case $\varepsilon<1$, the functions from Example \ref{ex:weighted_2} are light-tailed versions
of the max-kernel and the sum kernel studied, e.g.\ in \cite{GHMM22}. There, the authors focus on the
case  $\varepsilon\in(1,2)$, which leads to a power law for the degree distribution. 
For our version, this distribution has finite exponential moments of all orders for $\varepsilon< 1$ and finite exponential moments of some orders for $\varepsilon=1$.
Another example is
$g(p,q)=(p\wedge q)^{-\delta}(p\vee q)^\varepsilon$ for given $\delta,\varepsilon>0$. Then 
  $\|d_\varphi\|_{\infty,1}<\infty$ if and only if $\varepsilon< 1+\delta$. In this case, $g$ is a light-tailed version of the preferential attachment kernel; see \cite{GHMM22}. 
Note that $g(p,q)\ge (p\vee q)^{\varepsilon-\delta}$. Therefore, just as in Example \ref{ex:weighted_2}, 
under the condition $\varepsilon<1+\delta$, we have
$\|d^{(n)}_\varphi\|_{\infty,2}<\infty$ for $n=\lceil(2-2(\varepsilon-\delta))^{-1}\rceil$. 
Another example is $g(p,q)=|p-q|^\varepsilon$ for some given $\varepsilon>0$. Then
$\|d_\varphi\|_{\infty,1}<\infty$ if and only if $\varepsilon< 1$. If  $\varepsilon<1$ then we have
$\|d_\varphi^{(n)}\|_{\infty,2}<\infty$ for $n=\lceil(2-2\varepsilon)^{-1}\rceil$, just as in Example \ref{ex:weighted_2}.
We believe that for most natural examples, the condition $\|d_\varphi\|_{\infty,1}$ 
(necessary for strong sharpness) implies 
$\|d^{(n)}_\varphi\|_{\infty,2}<\infty$ for some finite $n\in\N$. 
The min kernel and the product kernel from \cite{GHMM22} do not satisfy $\|d_\varphi\|_{\infty,1}<\infty$.
\end{remark}

\subsection{Susceptibility mean-field bound}

The following Proposition is a refinement of \cite[Lemma 2.3]{CaicedoDickson24} 
and \cite[Proposition 2.1]{DicksonHeydenreich22}.

\begin{proposition}\label{p:equality_t_T} Suppose that $n\in\N$ satisfies \eqref{c:main}. 
Then $t_T=t^{(r)}_T$ for all $r\in[1,\infty]$.
\end{proposition}
\begin{proof} We exclude the trivial case $\|d_\varphi\|_{\infty,1}=0$, since
then  $t_T=t^{(r)}_T=\infty$. Assume that $\|d_\varphi\|_{\infty,1}>0$.
We only need to show that $t_T=t^{(\infty)}_T\ge t^{(1)}_T$. 
By Corollary \ref{c:clus_dom} we have
\begin{align*}
c^*(t)  \le  c^*_{\le n-1}(t)
+\esssup_{p\in\BM} \BE_t \int  \BE_t\big[ |V^{(x,q)}(\xi^{[n]})| \bigm\vert C^{(0,p)}_{\le n}\big]\,V^{(0,p)}_n(d(x,q)).
\end{align*}
By the definition of $\xi^{[n]}$ and stationarity, the conditional expectation in the above integral
equals $c^{(0,q)}(t)$. Therefore, we obtain from the definition \eqref{e:vpq} of density
$v^{p,q}_n(t)$ that
\begin{align*}
c^*(t)  \le  c^*_{\le n-1}(t)
+\esssup_{p\in\BM} \int  c^{(0,q)}(t) v^{p,q}_n(t)\,\BQ(dq).
\end{align*}
It follows that $c^*(t)\le c^*_{\le n-1}(t)+\bar{c}(t)\|v_n(t)\|_{\infty,\infty}$
and since $\bar{c}(t)\ge 1$ we obtain 
\begin{align}\label{i:Delta_n}
c^*(t)\le \Delta_n(t) \bar{c}(t).
\end{align}
By assumption \eqref{c:main} and Lemma \ref{l:Delta_n} we have $\Delta_n(t)<\infty$, 
so the asserted inequality $t_T\ge t^{(1)}_ T$ follows.
\end{proof}

\begin{proposition}[Susceptibility mean-field bound] \label{p:susc_mean_field_bound}
    Suppose $n\in\N$ satisfies \eqref{c:main}. Then
    \begin{align}\label{e:susc_mean_field_bound}
        \|c(t)\|_r\ge \frac{t_T}{\Delta_n(t)(t_T-t) } 
    \end{align}
    for $t<t_T$ and for all $r\in[1,\infty]$.
\end{proposition}
\begin{proof}
To prove this result, we need to recall some of the notation from \cite{CaicedoDickson24}. Define 
\begin{align*}
T_t(p,q)=\int \BP_t\big((0,p)\leftrightarrow(x,q) \text{ in } \xi^{(0,p),(x,q)}\big) \, dx
\end{align*}
and let the integral operator $\mathcal{T}_t$ act as
$(\mathcal{T}_t f)(p)=\int T_t(p,q)f(q)\,\BQ(dq)$, for every
square-integrable function $f$ on $\BM$ (with respect to the
probability measure $\BQ$). Let
$t_o=\inf\{ t>0 : \|\mathcal{T}_t\|_{op}=\infty\}$, where, as usual, $\|\cdot\|_{op}$ refers to the operator norm for a linear operator on a Banach space.

By \cite[Lemma 2.3]{CaicedoDickson24} (which applies without
Assumption D in that paper)  and Proposition
\ref{p:equality_t_T} we have $t_o=t_T$. From \cite[Lemma
3.2]{DicksonHeydenreich22}, we know
$\|\mathcal{T}_t\|_{op}\le \|T_t\|_{\infty,1}$ (this is proven by
Schur’s test). Since $c^{(0,p)}(t)=1+t \int T_t(p,q) \,
\BQ(dq)$ we have $c^*(t)=1+t\|T_t\|_{\infty,1}$. By
\cite[Theorem 2.5]{CaicedoDickson24} (which again does not require 
Assumption D) we have
$\|\mathcal{T}_t\|_{op}\ge (t_0-t)^{-1}$ for $t\in(0,t_0)$. Hence
$\|T_t\|_{\infty,1}\ge (t_T-t)^{-1}$ and
$c^*(t)\ge t_T (t_T-t)^{-1}$ for $t\in(0,t_T)$. For $r=1$, the asserted result now follows from the inequality \eqref{i:Delta_n}. The general case $r\ge 1$ follows from H\"older's inequality.
\end{proof}

\subsection{Strong sharpness of the phase transition}

Analogously to \cite{CaicedoDickson24}, we introduce a continuous and
mark-dependent analogy of the magnetization originally introduced by
Aizenman and Barsky \cite{AizBar87}.  Let $\gamma\in(0,1)$ be a
parameter with which we enrich the marked RCM by adding to each vertex
a uniform $(0,1)$ (Lebesgue) label (independent of everything else),
and let $\BP_{t,\gamma}$ denote the resulting probability measure. A
vertex $x \in\eta$ is called a {\em ghost} vertex if its label is at
most $\gamma$, and we write $x \in \cG$. Similarly, we write
$x \leftrightarrow \cG$ if $x$ is connected to a ghost vertex. We
define magnetization as follows
\begin{align}
\label{eqn:define_magnetization}
    M(t,\gamma,p) := \BP_{t,\gamma}( (0,p) \leftrightarrow \cG \text{ in } \xi^{(0,p)}).
\end{align}
In accordance with our previous notation, we
use the $L^r(\BQ)$-norms for $r\in[1,\infty]$ to define
\begin{equation}
\widebar{M}(t,\gamma):=\|M(t,\gamma,\cdot)\|_1, \quad M^*(t,\gamma) := \|M(t,\gamma,\cdot)\|_\infty.
\end{equation}

Recall the definition of $\xi^{Q_0}$ and $C^{Q_0}$ in Subsection \ref{sub:notationstatRCM}.
We assume that $Q_0$ is also independent of the labels.
Note that
$\widebar{M}(t,\gamma)= \BP_{t,\gamma}(\text{$(0,Q_0) \leftrightarrow \cG$ in $\xi^{Q_0}$})$.  
We will also need the following functions. For
$t\in\R_+$, we define 
\begin{align}
\bar{c}_{f}(t):= \BE_t |C^{Q_0}|\I\{ |C^{Q_0}| < \infty\} = \sum_{n\in\N}n \BP_t(|C^{Q_0}| = n),
\end{align}
and for $\gamma\in(0,1)$ we also define the ``ghost-free'' mean size of the cluster of a typical vertex
\begin{equation}
\bar{c}(t,\gamma) = \BE_{t,\gamma}|C^{Q_0}|\I\{ C^{Q_0} \cap \cG = \emptyset\}.
\end{equation}

It is easy to relate the above function to the mean size of the finite cluster of a typical vertex and 
the percolation probability $\bar\theta(t)=\BP_{t,\gamma}\big(\big|C^{Q_0}\big|=\infty\big)$.

\begin{lemma}\label{l:gamma_to_0}
    Let $t\in\R_+$. Then
    \begin{align*}
        \lim_{\gamma \to 0}\widebar{M}(t,\gamma) = \bar\theta(t), \quad
        \lim_{\gamma \to 0}\bar{c}(t,\gamma) = \bar{c}_{f}(t).
    \end{align*}
\end{lemma}
\begin{proof} The proof is direct and follows the same lines as in \cite{CaicedoDickson24}. 
For $\gamma \in(0,1)$ we have 
    \begin{align}
         \widebar{M}(t,\gamma) &= 1 - \BP_{t,\gamma}( (0,Q_0) \nleftrightarrow \cG \text{ in } \xi^{Q_0}) 
         = 1 - \sum_{n \in \N}\BP_{t,\gamma}(C^{Q_0} \cap \cG = \emptyset, |C^{Q_0}| = n) \nonumber \\
         &= 1 - \sum_{n \in \N}(1 - \gamma)^n\BP_t(|C^{Q_0}|=n)= 1 - \BE_t  (1 - \gamma)^{|C^{Q_0}|}.
         \label{e:mag_expression}
    \end{align}
Letting $\gamma \to 0$, we obtain the first result from monotone convergence.
It is also easy to see that the function
$\gamma\mapsto \widebar{M}(t,\gamma)$ is analytic on $(0,1)$ for all $t\in\R_+$.

For $\gamma>0$ we obtain from
our independence assumption that $|C^{Q_0}|<\infty$
a.s.\ on the event $\{C^{Q_0} \cap \cG = \emptyset\}$. 
We then have 
\begin{align}
\bar{c}(t,\gamma)& = \sum_{n\in\N} n \BP_{t,\gamma}(|C^{Q_0}|=n,C^{Q_0} \cap \cG = \emptyset) \nonumber \\ 
    &= \sum_{n \in\N} n (1 - \gamma)^n \BP_t(|C^{Q_0}|=n).
    \label{e:mag_der_gamma}
\end{align}
As $\gamma \to 0$, we obtain the second result from monotone convergence.
\end{proof}

We will also need the following lemma on differential inequalities for the magnetization. 

\begin{lemma}[Aizenman-Barsky differential inequalities on the magnetization] \label{l:diff_ineqs_mag} 
Let $\gamma \in (0,1)$ and $t>0$ and assume that $\| d_\varphi\|_{\infty,1}<\infty$. Then we have
\begin{enumerate}
\item[{\rm (i)}]  $\frac{\partial \widebar{M}(t,\gamma)}{\partial t} 
\leq (1-\gamma) \| d_\varphi\|_{\infty,1} M^*(t,\gamma) \frac{\partial \widebar{M}(t,\gamma)}{\partial\gamma}$.
\item[{\rm (ii)}] $\widebar{M}(t,\gamma) \leq 
\gamma \frac{\partial \widebar{M}(t,\gamma)}{\partial\gamma} 
+ \|M(t,\gamma)\|_2^2 + t M^*(t,\gamma) \frac{\partial \widebar{M}(t,\gamma)}{\partial t}$. 
\end{enumerate}
\end{lemma}
\begin{proof}
The second inequality follows from the Leibniz differentiation
    rule and \cite[Lemma 3.8]{CaicedoDickson24}.  
The first inequality (with $1/t$ instead of $\|d_\varphi\|_{\infty,1}$) is asserted in the same lemma.
However, one of the inequalities in the proof does not seem to be correct. Therefore, we present here an alternative argument.
We combine the Margulis--Russo formula in infinite volume from \cite[Theorem 10.8]{ChLast25}
with the Leibniz differentiation rule to obtain from  \eqref{e:mag_expression} that
\begin{align*}
       \frac{\partial \widebar{M}(t,\gamma)}{\partial t}
&=\BE_t \int  ((1 - \gamma)^{|C^{Q_0}|}-(1 - \gamma)^{|C^{Q_0}(\xi^{(0,Q_0),(x,p)})|}) 
\I\{(x,p)\in C^{Q_0}(\xi^{(0,Q_0),(x,p)})\} \, \lambda (d(x,p))\\
&= \BE_t \int  (1 - \gamma)^{|C^{Q_0}|} (1-(1 - \gamma)^{|C^{(x,p)}(\xi^{(0,Q_0),(x,p)}-C^{Q_0})|}) \I\{(x,p)\in C^{Q_0}(\xi^{(0,Q_0),(x,p)})\} \, \lambda (d(x,p)),
\end{align*}
where $\xi^{(0,Q_0),(x,p)}-C^{Q_0}$ is the random graph arising from $\xi^{(0,Q_0),(x,p)}$ by removing all vertices from $C^{Q_0}$ along with emanating edges.
Fix $(x,p)\in\R^d\times\BM$ for the moment. The conditional distribution of 
$\xi^{Q_0,(x,p)}$ given $C^{Q_0}$ is that of a random graph, which can be constructed in two steps as follows. Take first a RCM (with connection function $\varphi$) based on  $V^{(0,Q_0)}(\xi^{(0,Q_0)})+\delta_{(x,p)}+\eta^0$, where 
$\eta^0$ is a Poisson process with intensity measure $K_{t\lambda}(C^{Q_0})$. 
Then remove all edges between points of $V^{(0,Q_0)}(\xi^{(0,Q_0)})$ and add instead the original edges of 
$C^{Q_0}$. This distributional identity can be proved similarly to \cite[Lemma 3.3]{HHLM23}.
In particular, the event  $\{(x,p)\in C^{Q_0}(\xi^{(0,Q_0),(x,p)})\}$ (which means that there is a direct connection between $(x,p)$ and
a vertex from $C^{Q_0}$ in $\xi^{(0,Q_0),(x,p)}$) and the random variable 
$|C^{(x,p)}(\xi^{(0,Q_0),(x,p)}-C^{Q_0})|$ are conditionally independent. 
Therefore,
\begin{align*}
       \frac{\partial \widebar{M}(t,\gamma)}{\partial t}
= \BE_t \int  (1 - \gamma)^{|C^{Q_0}|} \varphi((x,p),C^{Q_0}) 
\BE_t\big[1-(1 - \gamma)^{|C^{(x,p)}(\xi^{(0,Q_0),(x,p)}-C^{Q_0})|}\mid C^{Q_0}\big] \, \lambda (d(x,p)).
\end{align*}
Furthermore, a stochastic monotonicity argument 
(see Proposition \ref{p2.2}) shows that, given $C^{Q_0}$, the random variable
$|C^{Q_0}(\xi^{(0,Q_0),(x,p)}-C^{Q_0})|$ is stochastically dominated by an independent
random variable with the distribution of $|C^{(x,p)}|$.
Therefore, the above can be bounded by
\begin{align*}
\BE_t \int  (1 - \gamma)^{|C^{Q_0}|} \varphi((x,p),C^{Q_0}) M(t,\gamma,p) \, \lambda (d(x,p))
&\le M^*(t,\gamma) \BE_t (1 - \gamma)^{|C^{Q_0}|} \varphi_\lambda (C^{Q_0})\\
&\le M^*(t,\gamma) \| d_\varphi\|_{\infty,1}  \BE_t |C^{Q_0}| (1 - \gamma)^{|C^{Q_0}|},
\end{align*}
where we have used the Bernoulli inequality to bound 
$\varphi_\lambda(C^{Q_0})$. This completes the proof.
\end{proof}

\begin{lemma}\label{l:M_upper_bound}
Let $t\in \R_+$, $n\in\N$ and $\gamma\in (0,1)$. Then
\begin{align}\label{e:M_upper_bound}
M^*(t,\gamma) \leq \Delta_n(t) \widebar{M}(t,\gamma).
\end{align}
\end{lemma}
\begin{proof}If $\Delta_n(t)=\infty$ or $\|d_\varphi\|_{\infty,1}=0$,
then inequality \eqref{e:M_upper_bound} is trivial. Therefore, we can assume that
  $\Delta_n(t)<\infty$ and $\|d_\varphi\|_{\infty,1}>0$.  
Fix $p\in\BM$. We have
    \begin{align*}
      M(t,\gamma,p)&=\BP_{t,\gamma}( (0,p) \leftrightarrow \cG \text{ in } \xi^{(0,p)}) 
= \BP_{t,\gamma}\bigg( \bigcup_{k\ge 0} \{ C_k^{(0,p)} \cap \cG \ne \emptyset\}\bigg) \\
      &\le  \BP_{t,\gamma}\left(  C_{\le n-1}^{(0,p)} \cap \cG \ne \emptyset\right) + \BP_{t,\gamma}\left( C_{\ge n}^{(0,p)} \cap \cG \ne \emptyset\right) \\
      &= 1- \BE_t (1-\gamma)^{|C_{\le n-1}^{(0,p)}|} + \BP_{t,\gamma}\left( C_{\ge n}^{(0,p)} \cap \cG \ne \emptyset\right) \\
      &\le  \gamma c_{\le n-1}^{(0,p)}(t) + \BE_{t,\gamma} \int \I\big\{ (x,q) \leftrightarrow \cG \text{ in } C_{\ge n}^{(0,p)}\big\} \, C^{(0,p)}_{n} (d (x,q)), 
    \end{align*}
where we have used the Bernoulli inequality for the last line. 
To treat the above second term $I_2$, say, we let $D^{(0,p)}_{\le n}$ denote the graph 
$C^{(0,p)}_{\le n}$, where all points (vertices) are marked by their labels, except those of 
$C^{(0,p)}_n$. Then
\begin{align*}
I_2=\BE_{t,\gamma} \int \BP \big((x,q) \leftrightarrow \cG \text{ in } C_{\ge n}^{(0,p)}\bigm\vert D^{(0,p)}_{\le n}\big) \, C^{(0,p)}_{n} (d (x,q)).
\end{align*}
By the spatial Markov property (in fact, we need the more refined \cite[Theorem 7.1]{ChLast25})
and stochastic monotonicity (as in Corollary \ref{c:clus_dom}) the conditional probability occurring above can be
bounded by $M(t,\gamma,p)$. Therefore,
\begin{align*}
I_2\le \BE_t \int M(t,\gamma,q) \, C^{(0,p)}_{n} (d (x,q))
=\int M(t,\gamma,q)v^{p,q}_n(t)\,\BQ(dq)\le \widebar{M}(t,\gamma) \| v_n(t)\|_{\infty,\infty},
\end{align*}
where the equality comes directly from the definition of the densities $v^{p,q}_n(t)$. 
Since $\gamma\le \widebar{M}(t,\gamma)$ we obtain the assertion.
\end{proof}

\begin{lemma}\label{l:M_lower_bound}
Let $t\in \R_+$, $n\in\N$, $\gamma\in (0,1)$ and assume  that $\bar{c}_{f}(t) = \infty$. Then
\begin{align}\label{e:M_lower_bound}
\widebar{M}(t,\gamma) \geq \sqrt{\frac{\gamma}{\Delta_n(t) + t \|d_\varphi\|_{\infty,1}\Delta^2_n(t)}}\geq \sqrt{\frac{\gamma}{1+ t \|d_\varphi\|_{\infty,1}}} \frac{1}{\Delta_n(t)}. 
\end{align}
\end{lemma}
\begin{proof}  We proceed similarly to the proof of \cite[Corollary 3.11]{CaicedoDickson24}.
If $\Delta_n(t)=\infty$, then the inequality \eqref{e:M_lower_bound} is trivial. Assume that
  $\Delta_n(t)<\infty$ and $\|d_\varphi\|_{\infty,1}>0$. Then we can use the preceding lemmas. Inserting the first inequality of
Lemma~\ref{l:diff_ineqs_mag} into the second, and using the simple
observation $\|M(t,\gamma)\|_2^2\le M^*(t,\gamma) \widebar{M}(t,\gamma)$
we obtain
\begin{align*}
\widebar{M}(t,\gamma) \leq \gamma \frac{\partial \widebar{M}(t,\gamma)}{\partial\gamma} 
+ M^*(t,\gamma) \widebar{M}(t,\gamma) + t\|d_\varphi\|_{\infty,1}(1-\gamma) (M^*(t,\gamma))^2 \frac{\partial \widebar{M}(t,\gamma)}{\partial\gamma}.
\end{align*}
Define $\widetilde\Delta_n(t):=t\|d_\varphi\|_{\infty,1} \Delta^2_n(t)$. Then by Lemma~\ref{l:M_upper_bound} we get 
\begin{align}\label{e:988}
\widebar{M}(t,\gamma) \leq \gamma \frac{\partial \widebar{M}(t,\gamma)}{\partial\gamma} 
+ \Delta_n(t) \widebar{M}(t,\gamma)^2 
+ (1-\gamma) \widetilde\Delta_n(t) \widebar{M}(t,\gamma)^2 \frac{\partial \widebar{M}(t,\gamma)}{\partial\gamma}.
    \end{align}
In the remainder of the proof, we fix $t$ and drop this argument from our notation.
By \eqref{e:mag_expression} the function $M$ (given by $\gamma\mapsto M(\gamma):= \widebar{M}(t,\gamma)$)
is strictly increasing and therefore has a differentiable inverse $M^{-1}$.
Then we can rewrite \eqref{e:988} for all $x$ in the range of $M$ as
\begin{align}\label{e:989}
x \leq \frac{M^{-1}(x)}{(M^{-1})'(x)} + \Delta_nx^2 
+ (1-M^{-1}(x)) \frac{\widetilde\Delta_nx^2}{(M^{-1})'(x)}.
    \end{align}
We have $ \frac{d}{dx}(x^{-1}M^{-1}(x)) = x^{-1}(M^{-1})'(x) - x^{-2}M^{-1}(x)$.
Multiplying \eqref{e:989} by $x^{-2}(M^{-1})'(x)$ we therefore obtain
\begin{align*}
  \frac{d}{dx}\big(x^{-1}M^{-1}(x)\big) 
\le \Delta_n (M^{-1})'(x)+(1-M^{-1}(x)) \widetilde\Delta_n
\le \Delta_n (M^{-1})'(x)+\widetilde\Delta_n.
\end{align*}
We wish to integrate this inequality on $[0,y]$ for some $y$ in the range of $M$.
To do so, we note that $M^ {-1}(0)=0$ and 
\begin{align*}
\lim_{x\to 0}\frac{M^{-1}(x)}{x}=(M^{-1})'(0)=\frac{1}{M'(0)}=\frac{1}{\bar{c}_f(t)}=0,
\end{align*}
where the penultimate identity follows from \eqref{e:mag_expression} and the final one
from our assumption $c_f(t)=\infty$. Therefore, we obtain
\begin{align*}
y^{-1}M^{-1}(y) \le \Delta_n M^{-1}(y)+\widetilde\Delta_ny,
\end{align*}
that is $\gamma/\widebar{M}(t,\gamma)\le \Delta_n(t) \gamma + \widetilde\Delta_n \widebar{M}(t,\gamma)$
for each $\gamma>0$. Since $\gamma\le \widebar{M}(t,\gamma)$, the first
inequality in \eqref{e:M_lower_bound} follows.
The second is then a consequence of $\Delta_n(t)\ge 1$, which is true by definition.
 \end{proof}

\begin{proposition}\label{p:per_lower_bound}
Let $t>t_T^{(1)}$ and $n\in\N$. If $\bar\theta(t_T^{(1)})=0$, then 
\begin{align}\label{e:per_lower_bound}
\bar\theta(t) \geq \frac{t-t_T^{(1)}}{2t \Delta_n(t)}.
\end{align}
\end{proposition}
\begin{proof}
  Note that if $\Delta_n(t)=\infty$, then the inequality
  \eqref{e:per_lower_bound} is trivial.  Suppose that
  $\Delta_n(t)<\infty$ and $\|d_\varphi\|_{\infty,1}>0$.  Then we have
by Proposition \ref{p:equality_t_T} that $t_T=t_T^{(1)}=t_T^{(\infty)}$. Let
  $t'>t_T$ and $t\in(t_T,t']$. Multiplying the second inequality of
  Lemma~\ref{l:diff_ineqs_mag} by
  $\gamma^{-1} \widebar{M}(t,\gamma)^{-1}$  and using simple observation
  $\|M(t,\gamma)\|_2^2\le M^*(t,\gamma) \widebar{M}(t,\gamma)$ gives
     \begin{align*}
       \frac{1}{\gamma} \leq \frac{1}{\widebar{M}(t,\gamma)} \frac{\partial \widebar{M}(t,\gamma)}{\partial\gamma} 
+ \frac{M^*(t,\gamma)}{\gamma} 
+ \frac{t M^*(t,\gamma)}{\gamma \widebar{M}(t,\gamma)} \frac{\partial \widebar{M}(t,\gamma)}{\partial t}.
    \end{align*}
    By Lemma~\ref{l:M_upper_bound} we get
     \begin{align*}
       \frac{1}{\gamma} &\leq  \frac{1}{\widebar{M}(t,\gamma)} \frac{\partial \widebar{M}(t,\gamma)}{\partial\gamma} 
+ \frac{\Delta_n(t)\widebar{M}(t,\gamma)}{\gamma} + \frac{t \Delta_n(t)}{\gamma } \frac{\partial \widebar{M}(t,\gamma)}{\partial t} \\
       &\leq  \frac{1}{\widebar{M}(t,\gamma)} \frac{\partial \widebar{M}(t,\gamma)}{\partial\gamma} 
+ \frac{\Delta_n(t')}{\gamma} \left(\widebar{M}(t,\gamma) + t \frac{\partial \widebar{M}(t,\gamma)}{\partial t}\right)\\
       &=  \frac{\partial \log (\widebar{M}(t,\gamma))}{\partial\gamma} 
+ \frac{\Delta_n(t')}{\gamma} \frac{\partial}{\partial t} \left( t \widebar{M}(t,\gamma) \right).
    \end{align*}
    We now integrate the above inequality over
    $(t,\gamma)\in [t_T,t']\times[\gamma_1,\gamma_2]$, where
    $0<\gamma_1\leq \gamma_2< 1$. Since all the integrands are non-negative, we can use Fubini's theorem to exchange the order of
    the integrals, and we will also use the properties of the function
    $\widebar{M}(t,\gamma)$, i.e.\ non-negativity and increasing in $t$ and
    $\gamma$. Therefore
\begin{align*}
       (t'-t_T) \log(\gamma_2/\gamma_1) 
&\leq\int_{t_T}^{t'} \log (\widebar{M}(t,\gamma_2)/\widebar{M}(t,\gamma_1))\, d t + \Delta_n(t') \int_{\gamma_1}^{\gamma_2}\frac{1}{\gamma} 
\left( t' \widebar{M}(t',\gamma) - t_T \widebar{M}(t_T,\gamma) \right) \, d \gamma \\
&\leq  (t'-t_T)\log (\widebar{M}(t',\gamma_2)/\widebar{M}(t_T,\gamma_1))+ \Delta_n(t')t' \widebar{M}(t',\gamma_2) \log(\gamma_2/\gamma_1).
\end{align*}
Dividing by $\log(\gamma_2/\gamma_1)$ we get   
\begin{align}\label{e:078}
       t'-t_T\leq (t'-t_T)\left(\frac{\log (\widebar{M}(t',\gamma_2))}{\log(\gamma_2)-\log(\gamma_1)} 
- \frac{\log(\widebar{M}(t_T,\gamma_1))}{\log(\gamma_2)-\log(\gamma_1)} \right)+ \Delta_n(t')t' \widebar{M}(t',\gamma_2).
\end{align}
Since $\bar\theta(t_T)=0$, the cluster of a typical vertex is finite
almost surely. Therefore $\bar{c}_f(t_T)=\BE_{t_T}|C^{Q_0}|$. 
Moreover, letting $t\uparrow t_T$ in \eqref{e:susc_mean_field_bound}, 
we obtain $\bar{c}(t_T)=\infty$. 
By Lemma \ref{l:M_lower_bound} we have for $0<\gamma_1<\gamma_2$ that
\begin{align*}
     - \frac{\log(\widebar{M}(t_T,\gamma_1))}{\log(\gamma_2)-\log(\gamma_1)} 
\le \frac{\log(\sqrt{1+t_T\|d_\varphi\|_{\infty,1}}\Delta_n(t_T))-\frac12\log (\gamma_1)}{\log(\gamma_2)-\log(\gamma_1)}.
\end{align*}
As $\gamma_1\to 0$, the above right-hand side tends to $\frac12$.
Hence it follows from \eqref{e:078} that
\begin{align*}
    \widebar{M}(t',\gamma_2)\geq \frac{t'-t_T}{2 t' \Delta_n(t')}. 
\end{align*}
Letting $\gamma_2\to 0$ and using Lemma \ref{l:gamma_to_0} concludes the proof.
\end{proof}

\begin{theorem}[Percolation mean-field bound]\label{t:per_mean_field_bound}
Let $\delta>0$. Suppose that $n\in\N$ satisfies \eqref{c:main}. Then 
\begin{align}\label{e:per_mean_field_bound}
    \|\theta(t)\|_r \geq \left(\frac{\bar\theta(t_T)}{\delta}
+\frac{\I\{\bar\theta(t_T)=0\}}{2t\Delta_n(t)}\right)(t-t_T),
\end{align}
for all $t\in[t_T, t_T + \delta]$ and $r\in[1,\infty]$.
\end{theorem}
\begin{proof}
By H\"older's inequality, we get $\|\theta(t)\|_r\ge \|\theta(t)\|_1\equiv\bar\theta(t)$ for all $r\in\left[1,\infty\right]$.
If $\bar\theta(t_T)>0$, then the result follows from the monotonicity of $\bar\theta(t)$. 
Otherwise, the result follows from Proposition \ref{p:per_lower_bound}. 
\end{proof}

\begin{corollary}[Sharpness of phase transition]\label{c:sharpness_of_ph_tran}
Under the assumptions of Proposition \ref{p:equality_t_T}, we have $t_c = t_T$.
\end{corollary}
\begin{proof}
As already noticed in Section \ref{s:perc} we always have $t_T\le t_c$. Let $t>t_T$. By
Theorem~\ref{t:per_mean_field_bound} there exists $A\in\cB(\BM)$ with $\BQ(A)>0$ and
$\theta^p(t)>0$ for all $p\in A$. Therefore $t\ge t_c$ and hence $t_T\ge t_c$. 
\end{proof}

\begin{remark}\label{r:generalmarke}\rm  
Let us consider here the hyperbolic counterpart of the stationary marked RCM; see \cite{Dickson25}.
In this case $\BX:=\BH^d\times\BM$, where $\BH^d$ is the $d$-dimensional hyperbolic space (with $d\ge 2$) 
as in Example \ref{ex:hyperbolic} and $\BM$ is as before. Assume that
$\lambda=\mathcal{H}^d\otimes\BQ$, where $\mathcal{H}^d$ denotes the Haar measure on $\R^d$
and $\BQ$ is a probability measure on $\BM$, also as before.
Assume that the connection function is given by $\varphi((x,p),(y,q))=\tilde{\varphi}(d_{\BH^d}(x,y),p,q)$ for
some measurable $\tilde\varphi\colon\R_+\times\BM^2\to[0,1]$ such that $\tilde{\varphi}(x,\cdot)$ is symmetric
for all $x\in\BH^d$. Fix a point $o\in \BH^d$.
Since the space $\BH^d$ is homogeneous, we can argue as
in Remark \ref{r:432} to see that
\begin{align*}
t_c=\sup\{t\ge 0: \BP_t(|C^{(o,p)}|<\infty)=0 \text{ for $\BQ$-a.e.\ $p$}\},\quad
t_T=\sup\big\{t\ge 0: \esssup_{p\in\BM} \BE_t |C^{(o,p)}|<\infty\big\}.
\end{align*}
Define $d_{\tilde\varphi}(p,q)=\int
\tilde{\varphi}(d_{\BH^d}(x,y),p,q)\,\mathcal{H}^d(dx)$ for
$(p,q)\in\BM^2$. Most of the results of this section
remain true in this setting, provided our assumptions are suitably modified 
(replacement of $d_\varphi$ by $d_{\tilde\varphi}$). Indeed, the geometry
of the ambient space $\R^d$ does not enter most of our arguments.
One exception is Proposition \ref{p:nottriv_t_c_2}. However, we believe that,
under the assumption $\|d_{\tilde\varphi}\|_{1,1} <\infty$, the proof of 
\cite[Proposition 1.1]{DicksonHeydenreich25}
can be extended to the marked case, just as in the Euclidean marked case.
We would then obtain that our condition \eqref{c:main}
(with hyperbolic $d_{\tilde\varphi}$) implies the strong sharp phase transition at $t_T=t_c\in(0,\infty)$, just as in the marked stationary Euclidean
case. 
Thanks to  \cite[Proposition 1.1]{DicksonHeydenreich25}
the condition from Proposition \ref{p:nottriv_t_c} is sufficient for $t_c<\infty$. 
\end{remark}

\bigskip
\noindent
{\bf Acknowledgements:}
This work was supported by the Deutsche Forschungsgemeinschaft
(DFG, German Research Foundation) through the SPP 2265, under grant
number LA 965/11-1.

\end{document}